\documentclass[12pt,a4paper]{amsart}
\pdfoutput=1
\usepackage{graphicx}
\usepackage{stmaryrd}
\usepackage{amssymb}
\usepackage{amsmath}
\usepackage{amsthm}
\usepackage{dsfont}
\usepackage{hyperref}
\usepackage{amsfonts}
\usepackage{amsrefs}
\usepackage{epsfig}
\usepackage{latexsym}
\usepackage{cite}

\usepackage{amsthm}
\usepackage{enumerate}
\usepackage{centernot}
\usepackage{fullpage}
\usepackage{mathrsfs}

\usepackage{color}

\newcommand{\blue}{\color{blue}}

\newtheorem{theorem}{Theorem}[section]
\newtheorem{corollary}[theorem]{Corollary}

\newtheorem{lemma}[theorem]{Lemma}

\theoremstyle{definition}

\newtheorem{remark}[theorem]{Remark}

\numberwithin{equation}{section}

\newcommand{\diam}{\text{\textnormal{diam}}}

\begin{document}
\title[Ising Magnetization Review]{Conformal Measure Ensembles and \\ Planar Ising Magnetization: A Review}

\author{Federico Camia}
\address{Division of Science, NYU Abu Dhabi, Saadiyat Island, Abu Dhabi, UAE \& Department of Mathematics, Faculty of Science, Vrije Universtiteit Amsterdam, De Boelelaan 1111, 1081 HV Amsterdam, The Netherlands.}
\email{federico.camia@nyu.edu}
\author{Jianping Jiang}
\address{Beijing Institute of Mathematical Sciences and Applications, No.11 Yanqi Lake West Road, Beijing 101407 \&  NYU-ECNU Institute of Mathematical Sciences at NYU Shanghai, 3663 Zhongshan
Road North, Shanghai 200062, China.}
\email{jianpingjiang11@gmail.com}
\author{Charles M. Newman}
\address{Courant Institute of Mathematical Sciences, New York University,
251 Mercer st, New York, NY 10012, USA, \& NYU-ECNU Institute of Mathematical
Sciences at NYU Shanghai, 3663 Zhongshan Road North, Shanghai 200062, China.}
\email{newman@cims.nyu.edu}


\begin{abstract}
We provide a review of results on the critical and near-critical scaling limit of the
planar Ising magnetization field obtained in the past dozen years. The results are presented in the framework of coupled loop and measure ensembles, and some new proofs are provided.
\end{abstract}


\maketitle

\section{Synopsis}

In \cite{cn09} the first and third authors introduced the concept of Conformal Measure Ensemble (CME) as the scaling limit of the collection of appropriately rescaled counting measures of critical FK-Ising clusters. They proposed to use a representation of the Ising magnetization field in terms of such a CME to study its existence, uniqueness and conformal properties in the critical scaling limit. Initial results and work in progress with Christophe Garban were presented by the first author at the Inhomogeneous Random Systems 2010 conference (Institut Henri Poincar\'e, Paris) and described in \cite{C12}. CMEs for Bernoulli and FK-Ising percolation were first constructed in \cite{CCK19}. The results on the two-dimensional Ising model discussed or conjectured in \cite{C12} have now been fully proved and have appeared in various papers by a combination of different authors \cite{cgn15,cgn16,CCK19,cjn20,cjn_coupling}. Those results, and more, were recently presented in a talk at the Inhomogeneous Random Systems 2020 conference (Institut Curie, Paris), which was the inspiration for the present paper, whose main goal is to review the results of \cite{cgn15,cgn16,cjn20,cjn_coupling} and present them in the unifying CME framework. While the results presented in this paper are not new, in some cases their formulation is somewhat different than what has previously appeared in the literature, and whenever we provide a detailed proof of a result, the proof is new.

\section{Introduction and historical remarks}

The Ising model was introduced by Lenz in 1920 \cite{Lenz20} to describe ferromagnetism, and is nowadays one of the most studied models of statistical mechanics.
The one-dimensional version of the model was studied by Ising in his Ph.D. thesis \cite{Ising24} and in his subsequent paper \cite{Ising25}, but it was not until
Peierls' and Onsager's famous investigations of the two-dimensional version that the model gained popularity. In 1936 Peierls \cite{Peierls36} proved that the
two-dimensional model undergoes a phase transition{\blue ;} then in 1941 Kramers and Wannier~\cite{KW41} located the critical temperature of the model defined on the
square lattice, and in 1944 Onsager \cite{Onsager44} derived its free energy. Since then, the two-dimensional Ising model has played a special role in the theory of
critical phenomena. Its phase transition has been extensively studied by both physicists and mathematicians, becoming a prototypical example and a test case for
developing ideas and techniques and for checking hypotheses.

Ferromagnetism is one of the most interesting phenomena in solid state physics; it refers to the tendency, observed in some metals such as iron and nickel, of the atomic
spins to become spontaneuosly polarized in the same direction, generating a macroscopic magnetic field. Such a tendency, however, is only present when the temperature
is lower than a metal-dependent characteristic temperature, called \emph{Curie temperature}. Above the Curie temperature the spins are oriented at random, producing
no net magnetic field. Moreover, as the Curie temperature is approached from either side, the specific heat of the metal appears to diverge.

The Ising model is a crude attempt to reproduce the behavior described above. Its one-dimensional version fails to do so, as already realized by Ising \cite{Ising24,Ising25},
but the two-dimensional version does exhibit a phase transition, as shown by Peierls \cite{Peierls36} and subsequently investigated by Kramers and Wannier \cite{KW41},
Onsager \cite{Onsager44} and many others.
In the most common version of the two-dimensional Ising model one associates a $\pm 1$ \emph{spin} variable to each vertex of a square grid
and then assigns to each spin configuration a probability derived from a Gibbs distribution that favors the alignment of neighboring spins.
The appeal of the two-dimensional version of the model stems from its simplicity and the fact that it yields to an exact treatment, which reveals a rich mathematical
structure. Its analysis has provided important tests for various fundamental aspects of the theory of critical phenomena such as the \emph{scaling hypothesis}, the emergence of scale
and conformal invariance at the \emph{critical point} marking a phase transition, and Landau's mean-field theory. All these different aspects of the theory of critical phenomena
find a natural interpretation in the \emph{renormalization group} philosophy, which asserts that the critical properties of a system do not depend on short-distance details
but only on the nature of long-distance fluctuations, suggesting a coarse-graining procedure that removes the short-distance features until one reaches the
\emph{correlation length} of the system, i.e., the characteristic length at which fluctuations become important and beyond which different parts of the system become uncorrelated.

The critical point is characterized by a diverging correlation length, as implied by Wu's celebrated result \cite{Wu66} showing that the correlation length of the two-dimensional Ising model diverges as the critical point is approached and the two-point
function between positions $x$ and $y$ decays like $|x-y|^{-1/4}$ at criticality.
This power-law behavior should be contrasted with the exponential decay that, away from the critical point, implies the existence of a finite correlation length
(see Figure \ref{phase_transition}).
\begin{figure}
		\begin{center}
\includegraphics[width=4cm]{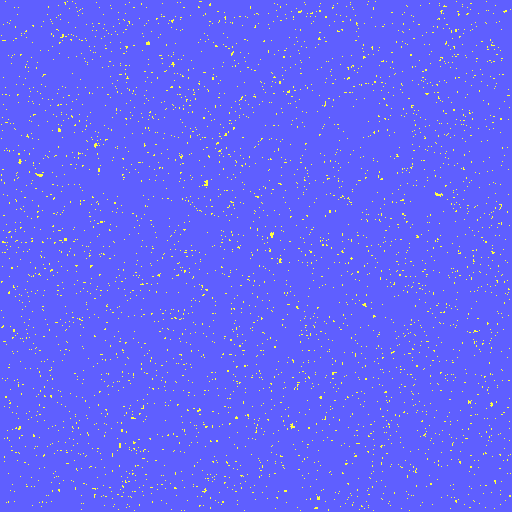}
\includegraphics[width=4cm]{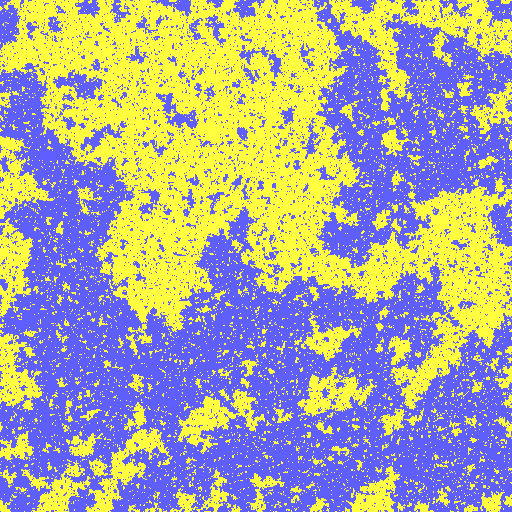}
\includegraphics[width=4cm]{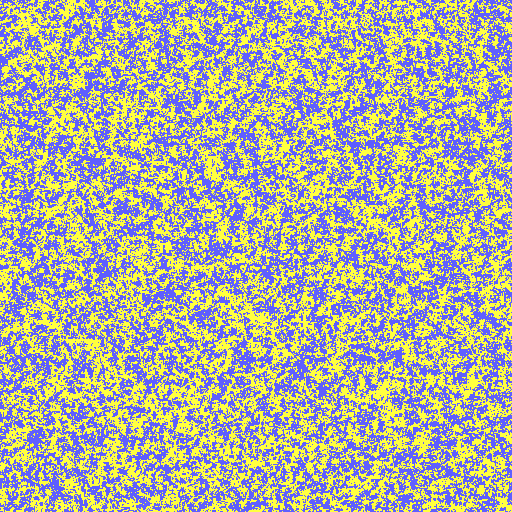}
		\end{center}
	\caption{The Ising model phase transition (courtesy of Wouter Kager). Blue (darker) and yellow (lighter) regions correspond to $+1$ and $-1$ spin values. The three images represent a typical low temperature, critical, high temperature configuration, respectively (from left to right).}
	\label{phase_transition}
\end{figure}

The renormalization group coarse-graining procedure can take the form of a \emph{continuum scaling limit} in which the mesh size of the grid on which the model is defined is sent to zero.
At the critical point, where the correlation length diverges, Smirnov \cite{Smirnov10} proved that certain observables of the two-dimensional Ising model have a well-defined scaling limit which is conformally invariant. This groundbreaking and much celebrated result confirms the emergence of conformal invariance and provides a link with conformal field theory.

The work of Chelkak, Hongler and Izyurov \cite{chi15} and of Camia, Garban and Newman \cite{cgn15} can be seen as the culmination of this line of research:
\cite{chi15} proves the existence and conformal invariance of the scaling limit of the $n$-point Ising correlation functions and \cite{cgn15} shows that in the scaling limit the Ising magnetization converges to a conformally invariant Euclidean field.

Conformal invariance is one of the most interesting features to emerge from the analysis of the scaling limit of critical models. Its emergence was predicted by Polyakov \cite{polyakov} and is discussed in \cite{BPZ84a,BPZ84b}. In two dimensions, conformal methods were applied extensively to Ising and Potts models, Brownian motion, the self-avoiding walk, percolation, and diffusion limited aggregation.
	
Moving away from criticality, one can modify the energy and hence the Gibbs distribution of the Ising model with a linear function of the spin variables. This models the presence of an external magnetic field that influences the alignment of the atomic spins. The model with an external field has never been solved exactly in dimension two or higher. However, in two dimensions
Zamolodchikov \cite{Zam89} proposed a solution of the model directly in the scaling limit in the shape of a field theory containing eight massive particles whose masses are related
to the exceptional Lie algebra $\text{E}_8$ (see \cite{BG11} for a ``journalistic'' account of this relation). In \cite{cgn16} Camia, Garban and Newman showed that a meaningful
scaling limit of this variant of the model can be obtained by scaling appropriately the external field with the lattice spacing in such a way that the correlation length should remain bounded
away from zero and infinity. In \cite{cjn20} the current authors showed that the resulting field theory has a mass gap, providing a first step in the direction of Zamolodchikov's theory.

The main goal of this paper is to review the results of \cite{cgn15,cgn16,cjn20,cjn_coupling} and present them in the unifying Conformal Measure Ensemble (CME) framework.
The CME idea and its usefulness in the study of scaling limits were first proposed in \cite{cn09}. CMEs for Bernoulli and FK-Ising percolation were first constructed in \cite{CCK19}.
The CME associated with FK-Ising percolation plays a crucial role in the proof of exponential decay of \cite{cjn20}.

We point out that significant progress continues to be made in other directions than those mentioned above. Even the recent literature on the Ising model is too vast to be surveyed in this paper, but we refer the interested reader to \cite{MM12} and references therein for a taste of recent results of a different flavor.

\section{The two-dimensional Ising model and its FK representation} \label{sec:Ising&FK}

We consider the standard Ising model on the square lattice $a{\mathbb Z}^2$ with (formal) \emph{Hamiltonian}
\begin{equation} \label{eq:hamiltonian}
{\bf H} = -\sum_{\{x,y\}} \sigma_x \sigma_y - H \sum_x \sigma_x \, ,
\end{equation}
where the first sum is over nearest-neighbor pairs in $a{\mathbb Z}^2$, the spin variables $\sigma_x, \sigma_y$ are $(\pm 1)$-valued and the external
field $H$ is in $\mathbb R$. For a bounded $\Lambda \subset {\mathbb Z}^2$, the \emph{Gibbs distribution} is given by
$\frac{1}{Z_{\Lambda}} \, e^{-\beta \, {\bf H}_{\Lambda}}$, where ${\bf H}_{\Lambda}$ is the Hamiltonian~(\ref{eq:hamiltonian}) with
sums restricted to vertices in $\Lambda$, $\beta \geq 0$ is the \emph{inverse temperature}, and the \emph{partition function} $Z_{\Lambda}$ is
the appropriate normalization needed to obtain a probability distribution. 

The critical inverse temperature is $\beta_c =\frac{1}{2} \, \log{(1+\sqrt{2})}$, and for all $\beta\leq\beta_c$
the model has a unique \emph{infinite-volume Gibbs distribution} for any value of the external field $H$, obtained as a weak limit of the
Gibbs distribution for bounded $\Lambda$ by letting $\Lambda \uparrow {\mathbb Z}^2$. For any value of $\beta\leq\beta_c$ and of
$H$, expectation with respect to the unique infinite-volume Gibbs distribution will be denoted by $\langle \cdot \rangle_{\beta,H}$.
At the \emph{critical point}, that is when $\beta=\beta_c$ and $H=0$, expectation will be denoted by $\langle \cdot \rangle_c$. 
By translation invariance, the \emph{two-point correlation} $\langle \sigma_x \sigma_y \rangle_{\beta,H}$ is a function only of $y-x$.
In particular, Wu \cite{Wu66} proved that $\langle\sigma_x\sigma_y\rangle_c$ decays like $|x-y|^{-1/4}$.

We want to study the random field $\Phi^{a,H}$ associated with the spins on the rescaled lattice $a {\mathbb Z}^2$ in the scaling limit $a \to 0$:
\begin{equation}\label{eqPhi}
\Phi^{a,H}:=\Theta_a\sum_{x\in a\mathbb{Z}^2}\sigma_x\delta_x,
\end{equation}
where $\delta_x$ is a unit Dirac point measure at $x$  and $\Theta_a$ is an appropriate scale factor. More precisely, for functions $f$ of bounded support on ${\mathbb R}^2$, we define
\begin{eqnarray} \label{eq:lat-field1}
\Phi^{a,H}(f) := \int_{{\mathbb R}^2} f(z) \Phi^{a,H}(z) dz & := & \int_{{\mathbb R}^2} f(z) \big[ \Theta_a \sum_{x \in a{\mathbb Z}^2} \sigma_x \delta(z-x) \big] dz \\
& = & \Theta_a \sum_{x \in a{\mathbb Z}^2} f(x) \sigma_{x},
\end{eqnarray}
with scale factor $\Theta_a$ proportional to
\begin{equation} \label{eq:Theta-first}
\Big( \sum_{x,y \in [0,1]^2 \cap a{\mathbb Z}^2} \langle \sigma_{x} \sigma_{y} \rangle_c \Big)^{-1/2}.
\end{equation}

The \emph{rescaled block magnetization} in a bounded domain $D$ is $M^a_D := \Phi^{a,H}({\mathbf I}_{D})$, where $\mathbf I$ denotes the indicator function.
It is a rescaled sum of identically distributed \emph{dependent} random variables. In the high-temperature regime, $\beta < \beta_c$, and with zero external
field, $H=0$, the dependence is sufficiently weak for the rescaled block magnetization to converge, as $a \to 0$, to a mean-zero Gaussian random variable (see, e.g.,~\cite{newman80} and references therein).
In that case, the appropriate scaling factor $\Theta_a$ is of order $a$, and the field converges to Gaussian white noise as $a \to 0$ (see, e.g.,~\cite{newman80}).
In the critical case, however, correlations are much stronger and extend to all length scales, so that one does not expect a Gaussian limit.

\subsection{Zero external field}

The FK representation of the Ising model with zero external field, $H=0$, is based on the $q=2$ random cluster measure $P^{FK}_p$
(see~\cite{grimmett-rcm-book} for more on the random cluster model and its connection to the Ising model). A spin configuration
distributed according to the unique infinite-volume Gibbs distribution with $H=0$ and inverse temperature $\beta\leq\beta_c$
can be obtained in the following way. Take a random-cluster (FK) bond configuration on the square lattice distributed according to
$P^{FK}_p$ with $p=p(\beta)= 1 - e^{-2\beta}$, and let $\{ {\mathcal C}^a_i \}$ denote the corresponding collection of FK clusters, where a cluster
is a maximal set of vertices of the square lattice connected via bonds of the FK bond configuration (see Figure~\ref{fig:loops}).
One may regard the index $i$ as taking values in the natural numbers, but it's better to think of it as a dummy countable index without
any prescribed ordering, like one has for a Poisson point process. Let $\{ \eta_i \}$ be ($\pm 1$)-valued, i.i.d., symmetric random  variables, and
assign $\sigma_x=\eta_i$ for all $x \in {\mathcal C}^a_i$; then the collection $\{ \sigma_x \}_{x \in a{\mathbb Z}^2}$ of spin variables is distributed
according to the unique infinite-volume Gibbs distribution with $H=0$ and inverse temperature $\beta$. When $\beta=\beta_c$, we will use
the notation $P^{FK}_c \equiv P^{FK}_{p(\beta_c)}$, and $E^{FK}_c$ for expectation with respect to $P^{FK}_c$.

\begin{figure}[!ht]
\begin{center}
\includegraphics[width=9cm]{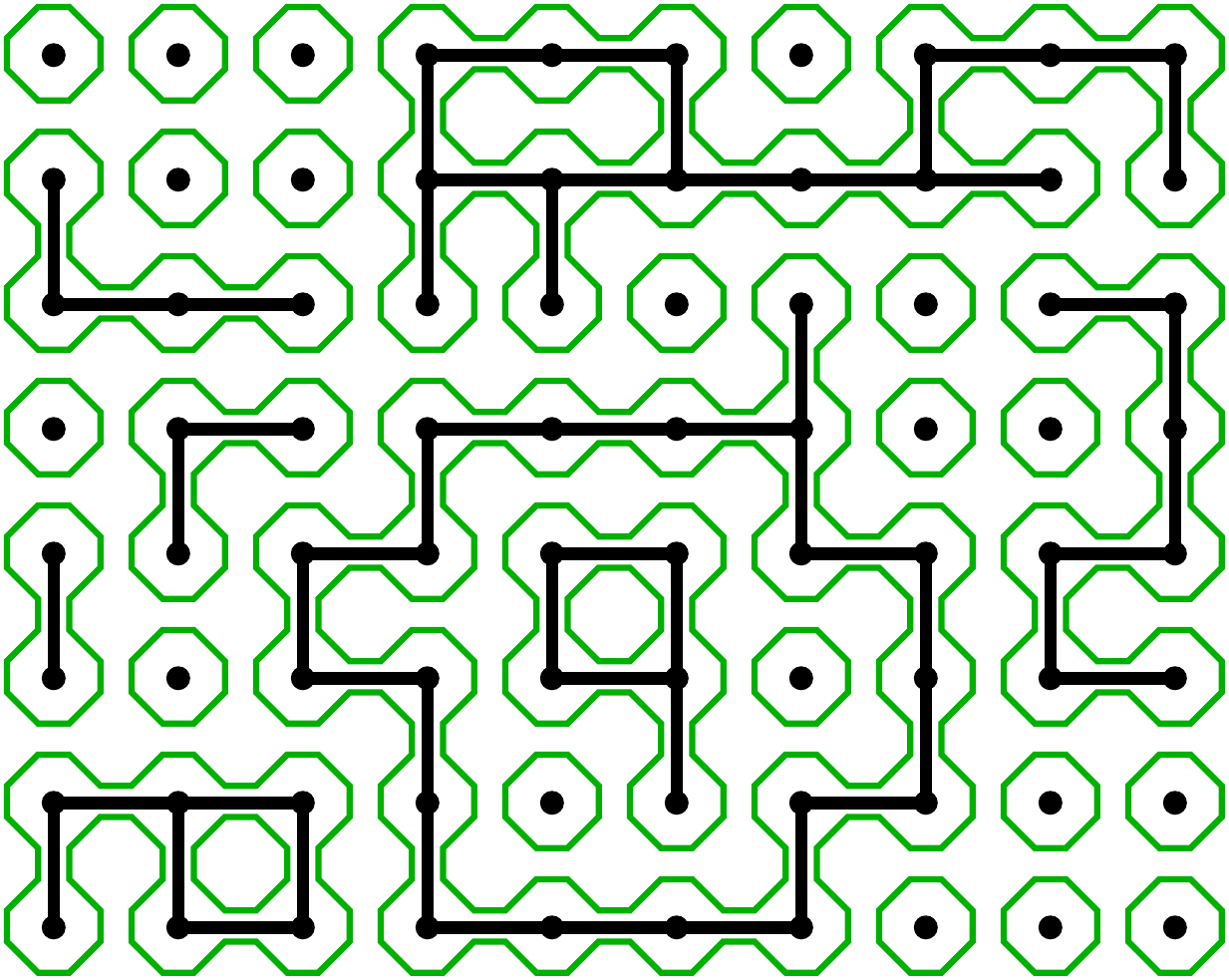}
\caption{Example of an FK bond configuration in a rectangular region (courtesy of Wouter Kager).
Black dots represent sites of ${\mathbb Z}^2$, black horizontal and vertical
edges represent FK bonds. The FK clusters are highlighted by lighter (green)
loops on the medial lattice.}
\label{fig:loops}
\end{center}
\end{figure}

A useful property of the FK representation is that, when $H=0$, the Ising two-point function can be written as
\begin{equation} \nonumber
\langle \sigma_x \sigma_y \rangle_{\beta,0} = P^{FK}_{p(\beta)}(x \text{ and } y
\text{ belong to the same FK cluster } {\mathcal C}^a_i) \, .
\end{equation}
As an immediate consequence, we have that $\Theta_a^{-2}$ is proportional to
\begin{equation} \label{eq:Theta}
\sum_{x,y \in [0,1]^2 \cap a{\mathbb Z}^2} \langle\sigma_x\sigma_y\rangle_c = \sum_{x,y \in [0,1]^2 \cap a{\mathbb Z}^2} E^{FK}_c \left[ \sum_i {\mathbf 1}_{x \in {\mathcal C}^a_i} {\mathbf 1}_{y \in {\mathcal C}^a_i} \right] = E^{FK}_c \left[ \sum_i | \hat{\mathcal C}^a_i |^2 \right],
\end{equation}
where $\hat{\mathcal C}^a_i$ is the restriction of ${\mathcal C}^a_i$ to $[0,1]^2$ and $|\hat{\mathcal C}^a_i|$ is the number of vertices of $a{\mathbb Z}^2$ in $\hat{\mathcal C}^a_i$.
(Note that $\hat{\mathcal C}^a_i$ need not be connected.) Using the FK representation and (\ref{eq:lat-field1}), we can write
\begin{equation} \label{eq:lattice-field}
\Phi^{a,0}(f) \stackrel{dist.}{=} \sum_i \eta_i \mu^a_i(f) \, ,
\end{equation}
where $\mu^a_i := \Theta_a \sum_{x \in {\mathcal C}^a_i}\delta(z-x)$ and the $\eta_i$'s are, as before, $(\pm 1)$-valued symmetric
random variables independent of each other and everything else. We can now easily see that $\Theta_a$ was chosen so that the second
moment of the rescaled block magnetization $M^a_{[0,1]^2}$,
\begin{equation} \label{eq:second-moment}
\left\langle \left[ \Phi^a({\mathbf I}_{[0,1]^2}) \right]^2 \right\rangle_c
= E^{FK}_c\Big[\sum_i \left(\mu^a_i({\mathbf I}_{[0,1]^2})\right)^2\Big] = \Theta^2_a E^{FK}_c \Big[ \sum_i |\hat{\mathcal C}^a_i|^2 \Big],
\end{equation}
is bounded away from $0$ and infinity.
Wu's celebrated result \cite{Wu66} on the decay of $\langle\sigma_x\sigma_y\rangle_c$ implies that, for the two-dimensional Ising model at the critical point,
we can take $\Theta_a=a^{15/8}$ so that $\mu^a_i := a^{15/8} \sum_{x \in {\mathcal C}^a_i}\delta(z-x)$ and $M^a_D := a^{15/8} \sum_{x \in a{\mathbb Z}^2 \cap D}\sigma_x$.


To each FK configuration we can associate a collection of loops on the medial lattice separating the FK clusters from the dual clusters,
where by dual clusters we mean maximal connected subsets of dual bonds, and a dual bond is an edge of the dual graph crossing perpendicularly
a primal edge that contains no FK bond. See Figure~\ref{fig:loops} for an example of an FK configuration with free boundary condition and the
corresponding collection of loops on the medial lattice.
We call a (random) collection of loops associated with an FK configuration in the way described above and shown in Fig.~\ref{fig:loops} a \emph{loop ensemble}.
We denote by $\{ \gamma^a_i \}$ the collection of all loops associated with the FK clusters $\{ {\mathcal C}^a_i \}$.
Each realization of $\{ \gamma^a_i \}$ can be seen as an element in a space of collections of loops with the Aizenman-Burchard metric~\cite{ab}.
(The latter is the induced Hausdorff metric on collections of curves associated to the metric on curves given by the infimum over monotone reparametrizations of the supremum norm.)
It follows from~\cite{ab} and the RSW-type bounds of~\cite{dhn} (see Section~ 5.2 there) that, as $a \to 0$, $\{ \gamma^a_i \}$ has
subsequential limits in distribution to random collections of loops in the Aizenman-Burchard metric. In the scaling limit, one gets
collections of nested loops that can touch (themselves and each other), but never cross.

\subsection{Non-zero external field} \label{sec:non-zero}

Appropriate FK representations exist also for the Ising model on the square lattice (or, indeed, on any graph) with an external field $H\neq0$. The ``standard'' one goes back to \cite{FK71} and involves adding a vertex $g$, called the ghost vertex, connected to all vertices of the square lattice and carrying either a plus or a minus spin, $\sigma_g= \pm 1$, in accordance with the sign of the external field (see Section~4.3 of \cite{FK71}). A spin configuration distributed according to the unique infinite-volume Gibbs distribution with $H\neq0$ and inverse temperature $\beta$ can be obtained by first taking a random-cluster (FK) bond configuration on the ``augmented'' square lattice obtained by adding the ghost spin as described above. In this case the distribution has two parameters: $p_1 = 1 - e^{-2\beta}$ for the edges between vertices of the square lattice, and $p_2 = 1 - e^{-2\beta H}$ for the edges connecting vertices of the square lattice to the ghost vertex. For each vertex $x$ in a cluster connected to $g$, one sets $\sigma_x=\sigma_g$. To clusters ${\mathcal C}^a_i$ not connected to $g$, one associates ($\pm 1$)-valued i.i.d. symmetric random variable $\eta_i$, and then sets $\sigma_x=\eta_i$ for all $x \in {\mathcal C}^a_i$.
Then the collection $\{ \sigma_x \}_{x \in a{\mathbb Z}^2}$ of spin variables is distributed according to the unique infinite-volume Gibbs distribution with $H\neq0$ and inverse temperature $\beta$.

This representation is not useful in the scaling limit $a \downarrow 0$, as we will discuss later (see Remark~\ref{rmk:standard} below). For this reason we introduce a somewhat different FK representation explained in detail in the next section. The latter is a direct consequence of the Edwards-Sokal coupling \cite{ES88} and it appears implicitly in \cite{CV16}. It was independently rediscovered and developed in \cite{cjn20} and \cite{cjn_coupling} where the authors observe how, contrary to the ``standard'' FK representation, it extends naturally to the near-critical scaling limit (see Remark \ref{rmk:standard} below).

%
%

\subsection{The Ising model in a bounded domain} \label{sec:bounded}

In this section we consider the Ising model on a finite graph $D^a = a{\mathbb Z}^2 \cap D$ with free or plus boundary condition, where $D$ is a bounded subset of ${\mathbb R}^2$. The boundary of $D^a$ will be denoted by $\partial D^a := \{ x \notin D^a : \exists y \in D^a \text{ with } |x-y|=a \}$. In order to define the \emph{random cluster model} associated with the Ising model (aka the FK-Ising model) on $D^a$, we introduce the following sets of edges:
\begin{eqnarray}
{\mathcal E}_i & = & \{ \{x,y\} : x,y \in D^a, |x-y|=a \} \\
{\mathcal E}_e & = & \{ \{x,y\} : x \in D^a, y \in \partial D^a, |x-y|=a \} \\
{\mathcal E} & = & {\mathcal E}_i \cup {\mathcal E}_e 
\end{eqnarray}
We call the edges in ${\mathcal E}_i$ and ${\mathcal E}_e$ 
\emph{internal} and \emph{external}, respectively. The distribution of the collection of spins $\sigma := \{ \sigma_x \}_{x \in D^a}$ for the Ising model on $D^a$ at inverse temperature $\beta$ with external field $H$ and free boundary condition is given by
\begin{equation} \label{free-bc}
P^{f}_{\beta,H}(\sigma) := \frac{1}{Z^f_{\beta,H}} \exp{\Big(\beta\sum_{ \{x,y\} \in \mathcal{E}_i} \sigma_x \sigma_y + \beta H \sum_{x \in D^a} \sigma_x \Big)},
\end{equation}
where
\begin{equation} \label{free-pf}
Z^f_{\beta,H} := \sum_{\sigma} \exp{\Big(\beta\sum_{ \{x,y\} \in \mathcal{E}_i} \sigma_x \sigma_y + \beta H \sum_{x \in D^a} \sigma_x \Big)}
\end{equation}
is the partition function of the model.
The distribution in the case of plus boundary condition is given by
\begin{equation} \label{plus-bc}
P^{+}_{\beta,H}(\sigma) := \frac{1}{Z^{+}_{\beta,H}} \exp{\Big(\beta\sum_{\{x,y\} \in \mathcal{E}_i} \sigma_x \sigma_y + \beta H \sum_{x \in D^a} \sigma_x + \beta\sum_{ \{x,y\} \in \mathcal{E}_e: x\in D^a} \sigma_x  \Big)},
\end{equation}
where
\begin{equation} \label{plus-pf}
Z^{+}_{\beta,H} := \sum_{\sigma} \exp{\Big(\beta\sum_{ \{x,y\} \in \mathcal{E}_i} \sigma_x \sigma_y + \beta H \sum_{x \in D^a} \sigma_x + \beta \sum_{ \{x,y\} \in \mathcal{E}_e: x\in D^a} \sigma_x\Big)}.
\end{equation}

The configuration space of the random cluster model is $\{0,1\}^{\mathcal E}$.
For each element $\omega$ of the configuration space and each edge $e=\{x,y\}$,
we say that edge $e$ is \emph{absent} or \emph{closed} if $\omega(e)=0$ and \emph{present} or \emph{open} if $\omega(e)=1$. We call $\omega$ a \emph{bond configuration}. A \emph{cluster} is a subset of $D^a \cup \partial D^a$ which is maximally connected using edges in ${\mathcal E}_i \cup {\mathcal E}_e$. We denote by $\mathscr{C}^a_D$ the collection of clusters restricted to $D^a$, that is
\begin{equation}
\mathscr{C}^a_D=\{\tilde{\mathcal{C}}^a\cap D^a: \tilde{\mathcal{C}}^a \text{ is a cluster in } D^a \cup \partial D^a\}\setminus\{\emptyset\}.
\end{equation}
To simplify the notation, in this section we write ${\mathcal C}_i$ for an element of $\mathscr{C}^a_D$ instead of ${\mathcal C}^a_i$ as in other sections. This shouldn't create any confusion since $a$ is fixed, but the reader should be aware of the fact that in later sections ${\mathcal C}$ will be used to denote an element of the collection $\mathscr{C}_D$ obtained from the scaling limit of $\mathscr{C}^a_D$ as $a \downarrow 0$ (see Theorem~\ref{thm:jointconv}).

For $\mathcal{C}\in\mathscr{C}_D^a$, we call $\mathcal{C}$ a \emph{boundary cluster} if $\mathcal{C}$ is the restriction of a cluster $\tilde{\mathcal{C}}$ (in  $D^a \cup \partial D^a$) which contains at least one element of $\partial D^a$; otherwise $\mathcal{C}$ is called an \emph{internal cluster}. The boundary is treated differently depending on the choice of boundary condition, as follows.

\noindent\emph{Free b.c.}: We set $\omega(e)=0$ for each $e \in {\mathcal E}_e$. In this case there are no boundary clusters.

\noindent \emph{Wired b.c.}: The vertices in $\partial D^a$ are identified and treated as a single vertex. Consequently, there is at most one boundary cluster. We denote by $\mathcal{C}_b$ this boundary cluster if it exists, and let $|\mathcal{C}_b|$ denote the number of vertices in $\mathcal{C}_b$.

We let ${\mathbf 1}(\omega) = | \{e \in {\mathcal E}: \omega(e)=1 \} |$ denote the number of open edges in the FK configuration, ${\mathbf 0}(\omega) = | \{e \in {\mathcal E}: \omega(e)=0 \} |$ denote the number of closed edges, and ${\bf C}(\omega)$ denote the number of \emph{internal} clusters.
With this notation, for any $0 \leq p \leq 1$, we define the random cluster measure
\begin{equation} \label{RC-distribution}
P^{FK}_{p}(\omega) := \frac{1}{Z^{FK}_{p}} p^{{\bf 1}(\omega)} (1-p)^{{\bf 0}(\omega)} 2^{{\bf C}(\omega)},
\end{equation}
where $Z^{FK}_{p} := \sum_{\omega \in \{0,1\}^{\mathcal E}} p^{{\bf 1}(\omega)} (1-p)^{{\bf 0}(\omega)} 2^{{\bf C}(\omega)}$ is the partition function of the model.
The random cluster measure corresponding to an Ising model on $D^a$ at inverse temperature $\beta$ with zero external field ($H=0$) is given by \eqref{RC-distribution} with the choice $p = 1 - e^{-2\beta}$. In this case we write
\begin{equation} \label{FK-distribution}
P^{FK}_{\beta,0}(\omega) := P^{FK}_{1-e^{-2\beta}}(\omega) = \frac{1}{Z^{FK}_{\beta,0}} (1-e^{-2\beta})^{{\bf 1}(\omega)} (e^{-2\beta})^{{\bf 0}(\omega)} 2^{{\bf C}(\omega)}
\end{equation}
where $Z^{FK}_{\beta,0}:=Z^{FK}_{1-e^{-2\beta}}$.

The following statements are immediate consequences of the Edwards-Sokal coupling of FK percolation and the Ising model \cite{ES88}.

\noindent\emph{Free b.c.}: A spin configuration on $D^a$ distributed according to $P^{f}_{\beta,0}$ can be obtained by
\begin{enumerate}
\item sampling an FK configuration according to \eqref{FK-distribution} with \emph{free boundary condition},
\item sampling an independent, ($\pm 1$)-valued, symmetric random variable $\eta_i$ for each FK cluster ${\mathcal C}_i \in \mathscr{C}^a_D$,
\item for each cluster ${\mathcal C}_i$, setting $\sigma_x=\eta_i$ for all $x \in {\mathcal C}_i$.
\end{enumerate}

\noindent\emph{Plus b.c.}: A spin configuration on $D^a$ distributed according to $P^{+}_{\beta,0}$ can be obtained by
\begin{enumerate}
\item sampling an FK configuration according to \eqref{FK-distribution} with \emph{wired boundary condition},
\item sampling an independent, ($\pm 1$)-valued, symmetric random variable $\eta_i$ for each internal cluster ${\mathcal C}_i$,
\item for each internal cluster ${\mathcal C}_i$, setting $\sigma_x=\eta_i$ for all $x \in {\mathcal C}_i$,
\item setting $\sigma_x=1$ for all $x$ belonging to the boundary cluster (if not empty).
\end{enumerate}

We note that coupled bond and spin configurations $(\omega,\sigma)$ generated by the Edwards-Sokal coupling are always \emph{compatible} in the sense that, if $x$ and $y$ belong to the same bond cluster, then $\sigma_x=\sigma_y$. If a $\omega$ and $\sigma$ are compatible, we write $\omega \sim \sigma$. This notation will be used in Corollary \ref{new-coupling} below and in its proof.

In the rest of the section, we will use $Z_{\beta,H}$ to denote either $Z^{f}_{\beta,H}$ or $Z^{+}_{\beta,H}$.
Moreover, letting $\tilde M^a_D := \sum_{x \in D^a} \sigma_x$ denote the (``\emph{bare}'') \emph{magnetization in $D^a$} and $E_{\beta,H}=\langle\cdot\rangle_{\beta,H}$ the expectation with respect to either $P^{f}_{\beta,H}$ or $P^{+}_{\beta,H}$, from \eqref{free-bc} and \eqref{plus-bc} we have that, for any suitable function $g$,
\begin{equation} \label{expectation-field}
E_{\beta,H}(g(\sigma)) = \frac{Z_{\beta,0}}{Z_{\beta,H}} E_{\beta,0}\Big( g(\sigma) e^{\beta H \tilde M^a_D} \Big).
\end{equation}
Taking the constant function $g(\sigma) \equiv 1$, this gives
\begin{equation} \label{pf-field}
Z_{\beta,H} = Z_{\beta,0} E_{\beta,0}\Big( e^{\beta H \tilde M^a_D} \Big)
= Z_{\beta,0} E^{FK}_{\beta,0}( e^{\beta H|\mathcal{C}_b|}\Pi_{i\neq b} \cosh(\beta H |{\mathcal C}_i|)),
\end{equation}
 where $|\mathcal{C}_b|=0$ if there is no boundary cluster, and the product is over all internal clusters in $\mathscr{C}^a_D$.

We will now present an FK representation for the Ising model with a non-zero external field that will be crucial when we discuss the near-critical scaling limit, in Section~\ref{sec:near-critical-sl}. 

\begin{corollary}[\cite{cjn20,cjn_coupling} ] \label{new-coupling}
Take $\beta>0$ and $H>0$.
For every $\omega \in \{0,1\}^{\mathcal E}$ with clusters $\{{\mathcal C}_i\}$, let
\begin{equation} \label{FK_external_field}
	P^{FK}_{\beta,H}(\omega) := \frac{e^{\beta H |\mathcal{C}_b|}\Pi_{i\neq b} \cosh(\beta H |{\mathcal C}_i|)}{E^{FK}_{\beta,0}( e^{\beta H |\mathcal{C}_b|}\Pi_{i\neq b} \cosh(\beta H |{\mathcal C}_i|))} P^{FK}_{\beta,0}(\omega).
\end{equation}
Furthermore, consider ($\pm 1$)-valued independent random variables $\{\eta_i\}$ such that
\begin{align} \label{C-to-g}
\eta_i=\begin{cases}
1 \text{ with probability } \frac{e^{\beta H |{\mathcal C}_i|}}{2 \cosh(\beta H |{\mathcal C}_i|)} = \tanh(\beta H |{\mathcal C}_i|) + \frac{1}{2}\big(1-\tanh(\beta H |{\mathcal C}_i|)\big) \\ -1 \text{ with probability } \frac{e^{-\beta H |{\mathcal C}_i|}}{2 \cosh(\beta H |{\mathcal C}_i|)} = \frac{1}{2}\big(1-\tanh(\beta H |{\mathcal C}_i|)\big)
\end{cases}
\end{align}

A spin configuration distributed according to $P^{f}_{\beta,H}$ can be obtained by sampling a configuration $\omega \in \{0,1\}^{\mathcal E}$ according to $P^{FK}_{\beta,H}$ on $D^a$ with free boundary condition and, for each cluster ${\mathcal C}_i$, setting $\sigma_x=\eta_i$ for all $x \in {\mathcal C}_i$.

A spin configuration distributed according to $P^{+}_{\beta,H}$ can be obtained by sampling a configuration $\omega \in \{0,1\}^{\mathcal E}$ according to $P^{FK}_{\beta,H}$ on $D^a$ with wired boundary condition, setting $\sigma_x=1$ for all $x$ in the boundary cluster (if not empty) and, for each internal cluster ${\mathcal C}_i$, setting $\sigma_x=\eta_i$ for all $x \in {\mathcal C}_i$.

The joint distribution of $\sigma$ and $\omega$ is
\begin{equation} \label{joint-dist}
P^{joint}_{\beta,H}(\sigma,\omega) := \frac{P^{FK}_{\beta,0}(\omega)}{E_{\beta,0}(e^{\beta H \tilde M^a_D})} \Big(\frac{1}{2}\Big)^{{\mathbf C}(\omega)} e^{\beta H \tilde M^a_D(\sigma)} {\mathbf
I}_{\{\omega \sim \sigma\}}.
\end{equation}
\end{corollary}

\begin{remark} \label{rmk:dimensions}
Corollary \ref{new-coupling} is valid in any dimension and a similar result holds for arbitrary graphs. These results also extend from Ising models to Potts models --- see the appendix of \cite{cjn_coupling}.
\end{remark}
\begin{remark} \label{rmk:standard}
As noted in Section \ref{sec:non-zero}, the ``standard'' FK representation for the Ising model with a non-zero external field involves adding a vertex $g$, called the ghost vertex, connected to all vertices of $D^a$ and carrying either a plus or a minus spin, $\sigma_g = \pm 1$, in accordance with the sign of the external field (see Section~4.3 of \cite{FK71}).
This makes it unsuitable for taking scaling limits, in which the lattice spacing is sent to zero and individual vertices have no meaning.
The representation discussed in Corollary \ref{new-coupling}, instead, relies only on ``macroscopic'' objects, namely the bond clusters.
As we will see in Section \ref{sec:near-critical-sl}, with the choice $H=ha^{15/8}$, the appearance of the combination $H|{\mathcal C}_i|$ in \eqref{FK_external_field} and \eqref{C-to-g} and of $H\tilde M^a_D=hM^a_D$ in \eqref{joint-dist} crucially allows us to make sense of this representation in the near-critical scaling limit.
\end{remark}
\begin{remark} \label{rmk:ghost}
	Equation \eqref{C-to-g} can be interpreted as saying that the spins in a bond cluster ${\mathcal C}_i$ ``follow'' the sign of the external field $H$ with probability $\tanh(\beta H |{\mathcal C}_i|)$ and otherwise (i.e., with probability $[1-\tanh(\beta H |{\mathcal C}_i|)]$) take either $+1$ or $-1$ with equal probability. As in the previous remark, one can think of a ghost vertex $g$ carrying spin $\sigma_g=1$. In this case, however, the ghost vertex connects to FK clusters instead of individual vertices.
\end{remark}

\begin{proof}
We begin by checking, with the help of the Edwards-Sokal coupling and of \eqref{pf-field}, that the marginal distribution induced by $P^{joint}_{\beta,H}$ on the bond configuration $\omega$ is $P^{FK}_{\beta,H}$:
\begin{eqnarray}
\sum_{\sigma} P^{joint}_{\beta,H}(\sigma,\omega) & = & \frac{P^{FK}_{\beta,0}(\omega)}{E_{\beta,0}(e^{\beta H \tilde M^a_D})} \Big(\frac{1}{2}\Big)^{{\mathbf C}(\omega)}  e^{\beta H  |{\mathcal C}_b|} \sum_{\{\eta_i=\pm 1:i\neq b\}} e^{\beta H \sum_i \eta_i |{\mathcal C}_i|} \\
& = & \frac{P^{FK}_{\beta,0}(\omega)}{E_{\beta,0}(e^{\beta H \tilde M^a_D})} e^{\beta H  |{\mathcal C}_b|} \Pi_{ i: i\neq b} \Big(\sum_{\eta_i=\pm 1} \frac{e^{\beta H \eta_i |{\mathcal C}_i|}}{2}\Big)
= P^{FK}_{\beta,H}(\omega).
\end{eqnarray}

Next, using again \eqref{pf-field} and the Edwards-Sokal coupling, we check that the marginal distribution induced by $P^{joint}_{\beta,H}$ on the spin configuration $\sigma$ is $P^{Ising}_{\beta,H}$:
\begin{eqnarray}
P^{Ising}_{\beta,H}(\sigma) & = & \frac{Z_{\beta,0}}{Z_{\beta,H}} P^{Ising}_{\beta, 0}(\sigma) e^{\beta H \tilde M^a_D(\sigma)} \\
& = & \frac{1}{E_{\beta,0}(e^{\beta H  \tilde M^a_D})} \sum_{\omega \sim \sigma} P^{FK}_{\beta,0}(\omega) \Big(\frac{1}{2}\Big)^{{\mathbf C}(\omega)} e^{\beta H \tilde M^a_D(\sigma)} \\
& = & \sum_{\omega} P^{joint}_{\beta,H}(\sigma,\omega).
\end{eqnarray}

Finally, we check the distribution of the $\eta_i$ variables, given by \eqref{C-to-g}.
To do this, we introduce the functions $S_i(\sigma,\omega)$ defined on pairs of compatible configurations $\{(\sigma,\omega): \omega \sim \sigma \}$ by letting $S_i(\sigma,\omega) = \sigma_x$ for any $x \in {\mathcal C}_i$, and note that $\tilde M^a_D(\sigma) = \sum_i S_i(\sigma,\omega)|{\mathcal C}_i|$.
Using \eqref{FK_external_field} and \eqref{joint-dist}, and letting ${\mathcal C}(x)$ denote the FK cluster of $x$, we have that
\begin{eqnarray}
P^{joint}_{\beta,H}(\sigma_x = \pm 1 | \omega) & = & \frac{P^{joint}_{\beta,H}( \sigma_x = \pm 1 \text{ and } \omega)}{P^{FK}_{\beta,H}(\omega)} \\
& = & \frac{1}{\Pi_{ i\neq b} \cosh(\beta H |{\mathcal C}_i|)} \sum_{\sigma: \sigma_x = \pm 1} \Pi_{ i\neq b} \Big( \frac{1}{2} e^{\beta H S_i(\sigma,\omega) |{\mathcal C}_i|} \Big) \\
& = & \frac{e^{\pm\beta H |{\mathcal C}(x)|}}{2\cosh(\beta H |{\mathcal C}(x)|)}.
\end{eqnarray}
This concludes the proof of the corollary.
\end{proof}

\section{The critical scaling limit} \label{sec:critical-s-lim}

In this section we discuss the critical scaling limit of FK-Ising clusters and of the Ising magnetization field. We restrict attention to bounded domains except where explicitly stated, as in Theorems \ref{thm:full-plane} and \ref{thm:conf-cov} and Remarks \ref{rem:joint-convergence} and \ref{rem:cme}.

\subsection{Conformal loop and measure ensembles} \label{sec:CLE&CME}

Theorem 1.1 of \cite{KS19} shows that, at the critical point, the (random) collection of loops $\{ \gamma^a_i \}$ coverges to a Conformal Loop Ensemble (CLE),
namely CLE$_{16/3}$, as $a \downarrow 0$.
CLEs are random collections of closed curves \cite{CN04,CN06,She09,SW12,Wer03} which provide a useful tool to encode and analyze the scaling limit geometry of,
for instance, Bernoulli and FK percolation clusters, Ising spin clusters, and loops in the $O(n)$ model. There is a one-parameter family of CLEs, CLE$_{\kappa}$,
indexed by a parameter $\kappa \in [8/3,8]$. For each $8/3 < \kappa \leq 8$, the loops of CLE$_{\kappa}$ locally ``look like'' SLE$_{\kappa}$ (see \cite{Sch00}).
At the extremes, CLE$_8$ almost surely consists of a single space-filling loop, which is the scaling limit of the outer boundary of the free uniform spanning tree (see \cite{LSW04}),
and CLE$_{8/3}$ almost surely contains no loops at all. When $8/3<\kappa<8$, the collection of loops in a CLE$_{\kappa}$ is almost surely countably infinite.
When $\kappa=6$, it is equivalent to the random collection of loops described in \cite{CN06}, where it was shown to arise as a scaling limit of the cluster boundaries of
critical site percolation on the triangular lattice.

The Schramm-Loewner Evolution (SLE) was introduced by Schramm \cite{Sch00} to describe the scaling limit of interfaces in critical models, where conformal invariance
is believed to emerge in such a limit. The introduction of SLE, combined with the work of Smirnov on percolation \cite{Smi01} and of Lawler, Schramm and Werner on the
loop erased random walk and the uniform spanning tree \cite{LSW04}, spurred a flurry of activity and led to substantial progress in the rigorous analysis of various critical
models, including percolation (see, e.g., \cite{SW01,CN04,lsw5,CN06,cfn06a,cfn06b,CN07} for some of the early papers) and the Ising model (see \cite{Smi07,cn09,Smirnov10}
for a sample of the early work).

Building on the convergence of $\{ \gamma^a_i \}$ to CLE$_{16/3}$ \cite{KS19}, Camia, Conijn and Kiss \cite{CCK19} provided the first construction of the CME for FK-Ising percolation.
As we will show below, the FK-Ising CME can be used to construct a conformal field $\Phi^0$ which is an element of an appropriate Sobolev space. $\Phi^0$ can be shown
to correspond to the scaling limit of the Ising magnetization at the critical point. We now introduce some notation and the relevant results from \cite{KS19} and \cite{CCK19}
before giving the construction of $\Phi^0$.

For any bounded domain $D \subset \mathbb{R}^2$, recall from Section \ref{sec:bounded} that $D^a=a\mathbb{Z}^2\cap D$ denotes its $a$-approximation. Let $L_1, L_2:[0,1]\rightarrow \bar{D}$ be two loops in the closure $\bar D$ of $D$. The distance between $L_1$ and $L_2$ is defined by
\[d_{\text{loop}}(L_1,L_2)=\inf\sup_{t\in[0,1]}|L_1(t)-L_2(t)|,\]
where the infimum is over all choices of parametrizations of $L_1,L_2$ from the interval $[0,1]$. The distance between two closed sets of loops, $F_1$ and $F_2$, is defined by the Hausdorff metric as follows:
\[d_{\text{LE}}(F_1,F_2)=\inf\{\epsilon>0: \forall L_1\in F_1,~\exists L_2\in F_2 \text{ s.t. } d_{\text{loop}}(L_1,L_2)\leq\epsilon\text{ and vice versa}\}.\]

The following theorem from \cite{KS19} establishes the convergence of the critical FK-Ising loop ensemble on the medial lattice (see  Figure \ref{fig:loops} above and Section 1.2.2 of \cite{KS19}) to CLE$_{16/3}$.
For a discrete domain $D^a = a{\mathbb Z}^2 \cap D$, we let $\Gamma^a_D$ denote the FK-Ising loop ensemble on the medial lattice.
\begin{theorem}[Theorem 1.1 in \cite{KS19}]\label{thmconvcle}
Consider critical FK-Ising percolation in a discrete domain $D^a$ with free or wired boundary condition. The FK-Ising loop ensemble $\Gamma^a_D$ converges in distribution
to CLE$_{16/3}$ in $D$ in the topology of convergence defined by $d_{\emph{LE}}$.
\end{theorem}

\begin{remark}\label{rem:bc}
A loop ensemble is a collection of loops on the medial lattice between FK clusters and dual clusters.
Since the critical point is the self-dual point for FK percolation, the critical FK loop ensembles for free and wired boundary conditions have the same distribution.
\end{remark}

For any configuration $\omega$ in critical FK percolation on $D^a$ with free or wired boundary condition, let $\mathscr{C}^a_D$ denote the set of clusters of $\omega$ in $D^a$.
For $\mathcal{C}^a_i\in\mathscr{C}^a_D$, let $\mu^a_i:=a^{15/8}\sum_{x\in\mathcal{C}^a_i}\delta_x$ be the rescaled (by $\Theta_a=a^{15/8}$) counting measure of
$\mathcal{C}^a_i$, and let $\mathscr{M}^a_D = \{\mu^a_i\}$.
For two collections, $\mathscr{C}_1$ and $\mathscr{C}_{2}$, of subsets of $\bar D$, the distance between $\mathscr{C}_1$ and $\mathscr{C}_{2}$ is defined by
\begin{equation}\label{eq:metric_clusters}
d_{\text{cl}}(\mathscr{C}_1, \mathscr{C}_{2}):=\inf\{\epsilon>0: \forall {\mathcal C}_1 \in \mathscr{C}_1~\exists {\mathcal C}_2 \in \mathscr{C}_{2} \text{ s.t. }d_H({\mathcal C}_1, {\mathcal C}_2)\leq\epsilon \text{ and vice versa}\},
\end{equation}
where $d_H$ is the Hausdorff distance.
Similarly, from two collections, $\mathscr{S}_1$ and $\mathscr{S}_{2}$, of measures on $D$, the distance between $\mathscr{S}_1$ and $\mathscr{S}_{2}$ is defined by
\begin{equation}\label{eq:metric}
d_{\text{meas}}(\mathscr{S}_1, \mathscr{S}_{2}):=\inf\{\epsilon>0: \forall \mu \in \mathscr{S}_1~\exists \nu \in \mathscr{S}_{2} \text{ s.t. }d_P(\mu, \nu)\leq\epsilon \text{ and vice versa}\},
\end{equation}
where $d_P$ is the Prokhorov distance. The following theorem establishes convergence of normalized counting measures; the result follows directly from Theorems 11 and 13
(see also Theorems 1 and 2 for simpler but slightly weaker versions), Theorem 14 and Lemma 9 of \cite{CCK19}.

\begin{theorem}[Theorems 11, 13, 14 and Lemma 9 of \cite{CCK19}]\label{thm:jointconv}
Let $D$ be a bounded, simply connected domain. As $a \downarrow 0$, $(\mathscr{C}^a_D, \mathscr{M}^a_D)$ converge jointly in distribution to $(\mathscr{C}_D, \mathscr{M}_D)$ where $\mathscr{C}_D$ is a collection of subsets of $\bar D$ and $\mathscr{M}_D$ is a collection of mutually orthogonal finite measures such that for every ${\mathcal C} \in \mathscr{C}_D$
there is a $\mu^0_{\mathcal C} \in \mathscr{M}_D$ with $\text{\emph{supp}}(\mu^0_{\mathcal C}) = {\mathcal C}$.
The topology of convergence is defined by $d_{\text{cl}} \times d_{\emph{meas}}$.

Moreover, the joint law of $(\Gamma^a_D, \mathscr{C}^a_D, \mathscr{M}^a_D)$ converges in distribution to the joint law of CLE$_{16/3}$ in $D$,
$\mathscr{C}_D$ and $\mathscr{M}_D$, with $\mathscr{C}_D$ and $\mathscr{M}_D$  measurable with respect to CLE$_{16/3}$ in~$D$.
\end{theorem}
We call the collection of measures $\mathscr{M}_D$ in the previous theorem a CME$_{16/3}$ in $D$ for its relation to CLE$_{16/3}$ in $D$.

A full-plane version of CLE can be constructed either by taking an increasing sequence of simply connected domains $D_n$ with $\cup_{n=1}^{\infty} D_n = {\mathbb C}$, and for each $n \in {\mathbb N}$ letting $\Gamma_n$ be a CLE$_{\kappa}$ in $D_n$, and then taking $\Gamma$ to be the limit of $\Gamma_n$ as $n \to \infty$ (see \cite{MWW16} for a detailed proof that the limit exists and does not depend on the sequence $(D_n)$), or equivalently by means of a branching SLE (see Section 2.3 of \cite{GMQ19}).
Theorem 3 of \cite{CCK19} shows that a full-plane CME$_{16/3}$ also exists. Combining these two results leads to the following theorem.
\begin{theorem} \label{thm:full-plane}
Let ${\mathbb P}_n$ denote the joint distribution of $(\Gamma_n,\mathscr{C}_n, \mathscr{M}_n)$ where $\Gamma_n$ and $\mathscr{M}_n$ are a CLE$_{16/3}$ and
its associated CME$_{16/3}$ in $[-n,n]^2$, respectively, and $\mathscr{C}_n$ is the collection of supports of the measures in $\mathscr{M}_n$.
There exists a probability measure $\mathbb P$ which is the full-plane limit of the probability measures ${\mathbb P}_n$ in the sense that, for every bounded domain $D$,
the restriction ${\mathbb P}_n|_D$ of ${\mathbb P}_n$ to $D$ converges to the restriction ${\mathbb P}|_D$ of $\mathbb P$ to $D$ as $n \to \infty$.
\end{theorem}

\begin{remark} \label{rem:joint-convergence}
	In treating the full-plane versions of CLE and CME it is convenient to consider the one-point (Alexandroff) compactification $\hat{\mathbb C}$ of $\mathbb C$, i.e., the Riemann sphere $\mathbb{\hat{C}}:=\mathbb{C}\cup\left\{ \infty\right\}$, and to replace the Euclidean distance with
\begin{equation} \label{new-dist}
\Delta(u,v) := \inf_{\varphi} \int\frac{|\varphi'(s)|}{1+|\varphi(s)|^2}ds,
\end{equation}
where we take the infimum over all continuous differentiable paths $\varphi(s)$ in $\mathbb C$ from $u$ to $v$ and	$|\cdot|$ denotes the Euclidean norm.
	Doing this ensures that the sequence $\Gamma_n$ has a unique limit in distribution as $n \to \infty$ by an application of Kolmogorov's extension theorem (see the proof of Theorem 3 of \cite{CCK19}). The joint convergence of $(\Gamma_n,\mathscr{C}_n,\mathscr{M}_n)$ follows from the fact that $\mathscr{C}_n$ and $\mathscr{M}_n$ are measurable with respect to $\Gamma_n$.
\end{remark}
\begin{remark} \label{rem:cme}
Theorem 4 of \cite{CCK19} shows that the full-plane collections of measures and their supports $(\mathscr{C},\mathscr{M})$ that are the limits of
$(\mathscr{C}_n,\mathscr{M}_n)$ are conformally invariant/covariant. This explains the name \emph{Conformal} Measure Ensemble (CME) for $\mathscr{M}$. Similar considerations apply to $\mathscr{M}_D$.
\end{remark}

\subsection{The magnetization field} \label{sec:m-field}

Combining \eqref{eq:lattice-field} with the results in the previous subsection, it is tempting to try to define a continuum magnetization field as
\begin{equation} \label{eq:cont-field}
\Phi^0_D(f) = \sum_{\mathcal{C}\in\mathscr{C}_D} \eta_{\mathcal{C}} \mu^0_{\mathcal{C}}(f),
\end{equation}
where $\mu^0_{\mathcal C} \in \mathscr{M}_D$ with $\text{supp}(\mu^0_{\mathcal C}) = {\mathcal C}$ and $\{ \eta_{\mathcal C} \}_{\mathcal{C}\in\mathscr{C}_D}$ is a collection of independent random variables such that $\eta_{\mathcal C}=1$ if ${\mathcal C}$ is a boundary cluster and $\eta_{\mathcal C}$ is a $(\pm 1)$-valued symmetric random variable otherwise.
However, due to scale invariance, even for bounded domains $D$ or functions $f$ of bounded support, the sum above contains infinitely many terms,
and the scaling covariance of the $\mu^0_{\mathcal{C}}$'s suggests that the collection $\{\mu^0_{\mathcal{C}}(f)\}_{\mathcal{C}\in\mathscr{C}_D}$
may in general not be absolutely summable.

In order to make sense of \eqref{eq:cont-field}, we introduce the $\varepsilon$-cutoff magnetization field
\begin{equation} \label{eq:cutoff-cont-field}
\Phi^0_{D,\varepsilon}(f) := \sum_{\mathcal{C}\in\mathscr{C}_D: \diam(\mathcal{C})>\varepsilon} \eta_{\mathcal{C}} \mu^0_{\mathcal{C}}(f),
\end{equation}
where the sum is now finite by Proposition 2.2 of \cite{cn09} whenever $D$ is bounded or $f$ has bounded support.

Before stating the next result, which provides a precise meaning for \eqref{eq:cont-field}, we need some preliminaries.
For a bounded domain $D$, let $\{ u_i \}$ denote an eigenbasis of the negative Laplacian (i.e., $-\Delta$) on $D$ with Dirichlet boundary condition with eigenvalues $0 \leq \lambda_1 \leq \lambda_2 \leq \ldots \uparrow \infty$. The functions $\{ u_i \}$ form an orthonormal basis of $L^2(D)$ and of the classical Sobolev space ${\mathcal H}_0^1(D)$, and they satisfy $\|u_i\|_{{\mathcal H}^1_0(D)}^2 = \lambda_i$. As a consequence, each $f \in {\mathcal H}^1_0(D)$ has a unique orthogonal decomposition $f=\sum_ia_iu_i$ with $\|f\|^2_{{\mathcal H}^1_0(D)} = \sum_i |a_i|^2 \lambda_i$.
Since $C^{\infty}_0 \subset {\mathcal H}^1_0(D)$, the same holds for each $f \in C^{\infty}_0(D)$.
Moreover, if $f\in C_0^\infty(D)$, then
\begin{equation}\label{e:ww}
\sum_i |a_i|^2 \lambda_i^\alpha <\infty, \quad \forall \alpha>0.
\end{equation}

To see why \eqref{e:ww} holds, following \cite{CGPR}, one can assume without loss of generality that $\alpha\geq 1$ is an integer. For such an $\alpha$ and for every $f\in C_0^\infty(D)$, one has that $(-\Delta)^\alpha f\in C_0^\infty(D)$, and consequently $(-\Delta)^\alpha f = \sum_i \langle (-\Delta)^\alpha f , u_i\rangle_{L^2(D)} u_i$, where the series converges in $L^2(D)$. Integration by parts yields moreover that $\langle (-\Delta)^\alpha f , u_i\rangle_{L^2(D)} = \langle  f , (-\Delta)^\alpha u_i\rangle_{L^2(D)}= \lambda_i^\alpha \langle f , u_i\rangle_{L^2(D)}$, from which we deduce that
\begin{equation}
\sum |a_i|^2 \lambda_i^\alpha = \langle (-\Delta)^\alpha f, f\rangle_{L^2(D)} \leq \|(-\Delta)^\alpha f\|_{L^2(D)} \,\cdot  \|f\|_{L^2(D)}<\infty,
\end{equation}
as claimed.

Given \eqref{e:ww}, one can define $\mathcal{H}_0^\alpha(D)$ to be the closure of $C_0^\infty(D)$ with respect to the norm $\|f\|^2_{\mathcal{H}_0^\alpha}(D) := \sum_i |a_i|^2 \lambda^\alpha_i$. The Sobolev space $\mathcal{H}^{-\alpha}(D)$ is then defined as the Hilbert dual of $\mathcal{H}_0^\alpha(D)$, that is, $\mathcal{H}^{-\alpha}(D)$ is the space of continuous linear functionals on $\mathcal{H}_0^\alpha(D)$, endowed with the norm $\|h\|_{\mathcal{H}^{- \alpha}(D) } := \sup_{f :  \|f\|_{\mathcal{H}^{\alpha}_0(D)} \leq 1} |h(f)|$. One has that $L^2(D) \subset \mathcal{H}^{-\alpha}(D)$. Also, the action of $h  = \sum_i a_iu_i\in L^2(D)$ on $f\in \mathcal{H}_0^\alpha(D)$ is given by $h(f) = \int_D h(z) f(z)\, dz$, and moreover $\|h\|^2_{\mathcal{H}^{-\alpha}(D)} = \sum_i \lambda^{-\alpha}_i |a_i|^2$.

\begin{lemma} \label{Lemma:Sobolev_convergence}
Let $D$ be a bounded, simply connected domain.
For every $\alpha >\frac{3}{2}$, the cutoff field $\Phi^0_{D,\varepsilon}$ converges as $\varepsilon \downarrow 0$ in second mean in the Sobolev space $\mathcal{H}^{-\alpha}(D)$, in the sense that there exists a $\mathcal{H}^{-\alpha}(D)$-valued random field $\Phi_D^0$ such that
\begin{equation}
\lim\limits_{\varepsilon \downarrow 0} E\big(\| \Phi^0_{D,\varepsilon} - \Phi_D^0 \|^2_{\mathcal{H}^{-\alpha}(D)}\big) = 0.
\end{equation}
\end{lemma}
\begin{proof}
We are going to show that $\Phi^0_{D,\varepsilon}$ is a Cauchy sequence in the Banach space of $\mathcal{H}^{-\alpha}(D)$-valued square-integrable random variables with norm $\sqrt{E\big(\| \cdot \|^2_{{\mathcal H}^{-\alpha}(D)}\big)}$ (see, e.g., \cite{LT91}).
The conclusion of the lemma then follows from the completeness of Banach spaces.

We will use the fact that $\Phi^0_{D,\varepsilon}$ is a.s.\ in $L^2(D)$ for every $\varepsilon>0$ (which follows from Theorem 2 of \cite{CCK19} and Proposition 2.2 of \cite{cn09}) together with the uniform bound $\|u_i\|_{L^\infty(D)}\leq c\lambda_i^{\frac{1}{4}}$ (see Theorem~1 of \cite{grieser}), Weyl's law \cite{weyl} (which says that the number of eigenvalues $\lambda_i$ that are less than $\ell$ is proportional to $\ell$ with an error of $o(\ell)$), and the argument in the proof of Proposition 2.1 of \cite{cn09} (see also the argument leading to equation~(10) in the proof of Theorem~7 of \cite{CCK19}).

With these ingredients we have that, $\forall \varepsilon>\varepsilon'>0$ and some $C,C'<\infty$,
\begin{eqnarray}
	E(\| \Phi^0_{D,\varepsilon} - \Phi^0_{D,\varepsilon'} \|^2_{\mathcal{H}^{-\alpha}(D)}) & = & \sum_i \frac{1}{\lambda_i^{\alpha}} E\Big[\big( (\Phi^0_{D,\varepsilon} - \Phi^0_{D,\varepsilon'})(u_i) \big)^2\Big] \\
	& \leq & \sum_i \frac{\| u_i \|^2_{L^{\infty}(D)}}{\lambda_i^{\alpha}} E\Big[\Big( \sum_{\mathcal{C}\in\mathscr{C}_D:\varepsilon'<\diam({\mathcal C})\leq\varepsilon} \eta_{\mathcal{C}} \mu^0_{\mathcal{C}}(\mathbf{I}_D) \Big)^2\Big] \\
	& \leq & \sum_i \frac{\| u_i \|^2_{L^{\infty}(D)}}{\lambda_i^{\alpha}} E\Big[ \sum_{\mathcal{C}\in\mathscr{C}_D:\diam({\mathcal C})\leq\varepsilon} \big(\mu^0_{\mathcal{C}}(\mathbf{I}_D) \big)^2 \Big] \\
	&\leq&\sum_i \frac{\| u_i \|^2_{L^{\infty}(D)}}{\lambda_i^{\alpha}}\lim_{a\downarrow0}a^{15/4}\sum_{x,y\in{D^a:|x-y|\leq\epsilon}}\langle\sigma_x\sigma_y\rangle_c\\
	& \leq & \sum_i \frac{\| u_i \|^2_{L^{\infty}(D)}}{\lambda_i^{\alpha}} C \varepsilon^{7/4} \\
	& \leq & C' \varepsilon^{7/4} \sum_i \frac{1}{\lambda_i^{\alpha-1/2}}. \label{L2-bound}
\end{eqnarray}
For any $\alpha>3/2$ the last expression is finite because of Weyl's law mentioned above, and it tends to zero as $\varepsilon \downarrow 0$, proving the claim.
\end{proof}

Lemma \ref{Lemma:Sobolev_convergence} shows that it makes sense to define the continuum magnetization field in a bounded domain $D$ as an element of ${\mathcal H}^{-\alpha}(D)$:
\begin{equation} \label{def:cont-field}
\Phi^0_D = \sum_{\mathcal{C}\in\mathscr{C}_D} \eta_{\mathcal{C}} \mu_{\mathcal{C}}
\end{equation}
where the right hand side is understood as the limit of $\Phi^0_{D,\varepsilon}$ as $\varepsilon \downarrow 0$.


The following theorem shows that the continuum magnetization field defined by \eqref{def:cont-field} is the scaling limit of the lattice magnetization field.
We note that Furlan and Mourrat \cite{FM17} have proved that, for any $\epsilon>0$, the lattice magnetization field is tight in Besov spaces of index $-1/8-\epsilon$ but not $-1/8+\epsilon$.
We remark that the scaling limit of the lattice energy field is rather different --- see Theorem 1.1 and Remarks 1.2 and 1.4 of \cite{Jia20}.
\begin{theorem} \label{Thm:Sobolev_convergence}
Consider an Ising model at the critical point ($\beta=\beta_c$ and $H=0$) with free or plus boundary condition on $D^a = D \cap a{\mathbb Z}^2$ where $D$ is a bounded, simply connected domain of ${\mathbb R}^2$. Then the lattice magnetization
\begin{equation}
\Phi^{a,0}_D := a^{15/8} \sum_{x\in D^a}\sigma_x\delta_x
\end{equation}
converges in distribution to $\Phi_D^0$ as $a \downarrow 0$. The convergence is in any Sobolev space $\mathcal{H}^{-\alpha}(D)$ with $\alpha >\frac{3}{2}$ in the topology induced by the norm $\| \cdot \|_{\mathcal{H}^{-\alpha}(D)}$.
\end{theorem}
\begin{proof}
To simplify the notation, in this proof we write $\Phi^a_D$ for $\Phi^{a,0}_D$.
We first show that $\Phi^a_D$ has subsequential limits in distribution in the topology induced by $\|\cdot\|_{{\mathcal H}^{-\alpha}(D)}$, for any $\alpha>3/2$.
We will make use of Rellich's theorem, which implies that ${\mathcal H}^{-\alpha_1}(D)$ is compactly embedded in ${\mathcal H}^{-\alpha_2}(D)$ for any $\alpha_1<\alpha_2$ and thus, in particular, that the closure of a ball of finite radius in ${\mathcal H}^{-\alpha_1}(D)$ is compact in ${\mathcal H}^{-\alpha_2}(D)$.

Given $\alpha>3/2$, let $\epsilon=\alpha-3/2>0$ and $\alpha'=3/2+\epsilon/2<\alpha$. A straightforward calculation, using some of the ingredients of the proof of Lemma~\ref{Lemma:Sobolev_convergence}, shows that
\begin{eqnarray}
\limsup_{a \downarrow 0} E\Big( \| \Phi^a_D \|^2_{{\mathcal H}^{-\alpha'}(D)} \Big) & = & \limsup_{a \downarrow 0} \sum_i \frac{1}{\lambda_i^{\alpha'}} E\Big(\big( \Phi^a_D(u_i) \big)^2 \Big) \\
& \leq & \Big(\sum_i \frac{\| u_i \|^2_{L^{\infty}(D)}}{\lambda_i^{\alpha'}}\Big) \limsup_{a \downarrow 0} E\Big( \big(\Phi^a_D(\mathbf{I}_D)\big)^2 \Big) \\
& \leq & c^2 \Big(\sum_i \frac{1}{\lambda_i^{1+\epsilon/2}}\Big) \limsup_{a \downarrow 0} E\Big( \big(\Phi^a_D(\mathbf{I}_D)\big)^2 \Big), \label{second-moment-bound}
\end{eqnarray}
where the last inequality follows from the bound in Theorem~1 of \cite{grieser} and the last expression is finite because of Weyl's law \cite{weyl} and the boundedness of the second moment of the rescaled magnetization (see the discussion leading to \eqref{eq:second-moment}).
This calculation, combined with Chebyshev's inequality and Rellich's theorem, implies that $\Phi^a_D$ is tight in ${\mathcal H}^{-\alpha}(D)$ for any $\alpha>3/2$.

Next we are going to show that all subsequential limits coincide with the field $\Phi^0_D$ defined by \eqref{def:cont-field}.
Let $\tilde{\Phi}^0_D$ denote any such subsequential limit obtained from a converging sequence $\Phi^{a_k}_D$. For any $\varepsilon>0$, using the Edwards-Sokal coupling, we can write
\begin{equation}
\Phi^{a_k}_D = \Phi^{a_k}_{D,\varepsilon} + \sum_{i:\diam({\mathcal C}_i)\leq\varepsilon} \eta_i \mu^{a_k}_i,
\end{equation}
where $\Phi^{a_k}_{D,\varepsilon} := \sum_{i: \diam(\mathcal{C}^a_i)>\varepsilon} \eta_i \mu^{a_k}_i$.
Since $E\Big( \| \Phi^a_{D,\varepsilon} \|^2_{{\mathcal H}^{-\alpha'}(D)} \Big) \leq E\Big( \| \Phi^a_D \|^2_{{\mathcal H}^{-\alpha'}(D)} \Big)$, the argument above shows that $\Phi^a_{D,\varepsilon}$ is tight in ${\mathcal H}^{-\alpha}(D)$ as $a \downarrow 0$ for any $\alpha>3/2$.
Proposition 2.2 of \cite{cn09} implies that, for any $f \in C^{\infty}_0(D)$, the number of elements in the sum defining $\Phi^{a}_{D,\varepsilon}(f)$ remains finite as $a \downarrow 0$; hence for every $f \in C^{\infty}_0(D)$, by Theorem \ref{thm:jointconv}, $\Phi^{a}_{D,\varepsilon}(f)$ converges in distribution to $\Phi^0_{D,\varepsilon}(f)$ as $a \downarrow 0$.
This, combined with the tightness of $\Phi^a_{D,\varepsilon}$ and with Lemma A.4 of \cite{CGPR}, shows that, as $k \to \infty$,
$\Phi^{a_k}_{D,\varepsilon}$ converges to $\Phi^{0}_{D,\varepsilon}$ in distribution in the topology induced by $\| \cdot \|_{{\mathcal H}^{-\alpha}(D)}$.
Therefore, $\sum_{i:\diam({\mathcal C}_i)\leq\varepsilon} \eta_i \mu^{a_k}_i$ also converges in distribution in the topology induced by $\| \cdot \|_{{\mathcal H}^{-\alpha}(D)}$
to some $\tilde{X}_{\varepsilon} \in {\mathcal H}^{-\alpha}(D)$, and $\Phi^{a_k}_D$ converges to $\tilde{\Phi}^0_D = \Phi^{0}_{D,\varepsilon}+\tilde{X}_{\varepsilon}$.
Moreover, a calculation analogous to that carried out in the proof of Lemma \ref{Lemma:Sobolev_convergence} shows that
\begin{eqnarray}
E(\| \tilde{X}_{\varepsilon} \|^2_{\mathcal{H}^{-\alpha}(D)}) & \leq & \limsup_{a_k \downarrow 0} E_{\beta_c,0}\Big(\Big\| \sum_{i:\diam({\mathcal C}_i)\leq\varepsilon} \eta_i \mu^{a_k}_i({\mathbf I}_{D}) \Big\|^2_{\mathcal{H}^{-\alpha}(D)}\Big) \\
& \leq & \sum_i \frac{\| u_i \|^2_{L^{\infty}(D)}}{\lambda_i^{\alpha}}  \limsup_{a_k \downarrow 0} E^{FK}_{\beta_c,0}\Big[ \sum_{i:\diam({\mathcal C}_i)\leq\varepsilon} \big(\mu^{a_k}_i({\mathbf I}_{D}) \big)^2 \Big] \label{small-diameter} \\
& \leq & C' \varepsilon^{7/4} \sum_i \frac{1}{\lambda_i^{\alpha-1/2}} \rightarrow 0 \text{ as } \varepsilon \to 0. \label{zero-limit}
\end{eqnarray}

We note that the processes $\Phi^{0}_D$, $\Phi^{0}_{D,\varepsilon}$, $\tilde{\Phi}^{0}_{D}$, and $\tilde{X}_{\varepsilon}$ are all measurable
with respect to CLE$_{16/3}$ in $D$ and have a joint distribution on the space of conformal loop and measure ensembles.
Hence, using the triangle inequality and Jensen's inequality, we have that
\begin{eqnarray}
E(\| \Phi^{0}_D - \tilde{\Phi}^{0}_{D} \|_{\mathcal{H}^{-\alpha}(D)}) & \leq &
E(\| \Phi^{0}_D - \Phi^{0}_{D,\varepsilon} \|_{\mathcal{H}^{-\alpha}(D)}) +
E(\| \tilde{X}_{\varepsilon} \|_{\mathcal{H}^{-\alpha}(D)}) \\
& \leq & E(\| \Phi^{0}_D - \Phi^{0}_{D,\varepsilon} \|_{\mathcal{H}^{-\alpha}(D)}) + \Big(C' \varepsilon^{7/4} \sum_i \frac{1}{\lambda_i^{\alpha-1/2}}\Big)^{1/2}. \label{bound-mean}
\end{eqnarray}

Markov's inequality now implies that, for any $\delta>0$,
$P\big( \| \Phi^0_D - \tilde{\Phi}^0_D \|_{\mathcal{H}^{-\alpha}(D)} > \delta \big)$ can be made arbitrarily small using the bound \eqref{bound-mean} together with \eqref{zero-limit} and Lemma \ref{Lemma:Sobolev_convergence}, and taking $\varepsilon$ small.
This concludes the proof of uniqueness of subsequential limits and of the theorem.
\end{proof}

Theorem \ref{Thm:Sobolev_convergence} and equation \eqref{eq:cont-field} provide a sort of continuum FK representation for the magnetization field.
A related result was recently proved by Miller, Sheffield and Werner~\cite{MSW17}. Theorem 7.5 of~\cite{MSW17} shows that forming clusters of CLE$_{16/3}$ loops
by a percolation process with parameter $p=1/2$ generates CLE$_{3}$, the Conformal Loop Ensemble with parameter $3$. CLE$_3$ describes the full scaling limit of
Ising spin-cluster boundaries \cite{BH19} while CLE$_{16/3}$, as already mentioned, describes the full scaling limit of FK-Ising cluster boundaries~\cite{KS19}.
We note that, although the magnetization can obviously be expressed using Ising spin clusters, as a sum of their signed areas, such a
representation does not appear to be useful in the scaling limit because the area measures of spin clusters scale with exponent $187/96$ by \cite{SSW09} while the magnetization scales with exponent $15/8$. 
The usefulness of the representation in terms of FK clusters is due to the fact that both the FK clusters and the magnetization need to be
multiplied by the same scale factor in the scaling limit to obtain meaningful nontrivial limits. That is not the case for the magnetization and the spin clusters.

To conclude this section we note that, using Theorem \ref{thm:full-plane}, it is possible to consider a full-plane version of the magnetization, which one can denote
\begin{equation} \label{def:full-plane-field}
\Phi^0 = \sum_{\mathcal{C}\in\mathscr{C}} \eta_{\mathcal{C}} \mu^0_{\mathcal{C}}
\end{equation}
where $\mathscr{C}$ denotes the full-plane collection of continuum clusters.
To be more precise, consider the magnetization field $\Phi^0_n := \Phi^0_{[-n,n]^2}$. Given $f \in C^{\infty}_0$, let $A_{f,k}$ denote the event that no continuum cluster intersects both the support of $f$ and the complement of $[-k,k]^2$.
Using Theorem \ref{thm:full-plane}, for any $k<n$, we can write the distribution function of $\Phi^0_n(f)$ as
\begin{align}
& {\mathbb P}_n(\Phi^0_n(f) \leq x) \nonumber \\
& \quad = {\mathbb P}_n(\Phi^0_n(f) \leq x \text{ and } A_{f,k}) + {\mathbb P}_n(\Phi^0_n(f) \leq x \text{ and } A_{f,k}^c) \\
& \quad = {\mathbb P}_n \big|_{[-k,k]^2}\Big(\sum_{\mathcal{C}\in\mathscr{C}:\mathcal{C}\cap\text{supp}(f)\neq\emptyset} \eta_{\mathcal{C}} \mu^0_{\mathcal{C}}(f) \leq x \text{ and } A_{f,k}\Big) \\
& \qquad \qquad + {\mathbb P}_n(\Phi^0_n(f) \leq x \text{ and } A_{f,k}^c). \nonumber
\end{align}
The fact that there is no infinite cluster in critical FK-Ising percolation implies that one can make the term
${\mathbb P}_n(\Phi^0_n(f) \leq x \text{ and } A_{f,k}^c)$ arbitrarily small by taking $n$ and $k$ sufficiently large.
Theorem \ref{thm:full-plane} implies that ${\mathbb P}_n \big|_{[-k,k]^2}\Big(\sum_{\mathcal{C}\in\mathscr{C}:\mathcal{C}\cap\text{supp}(f)\neq\emptyset} \eta_{\mathcal{C}} \mu^0_{\mathcal{C}}(f) \leq x \text{ and } A_{f,k}\Big)$ converges to ${\mathbb P} \big|_{[-k,k]^2}\Big(\sum_{\mathcal{C}\in\mathscr{C}:\mathcal{C}\cap\text{supp}(f)\neq\emptyset} \eta_{\mathcal{C}} \mu^0_{\mathcal{C}}(f) \leq x \text{ and } A_{f,k}\Big)$ as $n \uparrow \infty$. This last probability converges to ${\mathbb P}\Big(\sum_{\mathcal{C}\in\mathscr{C}:\mathcal{C}\cap\text{supp}(f)\neq\emptyset} \eta_{\mathcal{C}} \mu^0_{\mathcal{C}}(f) \leq x\Big) =: {\mathbb P}(\Phi^0(f) \leq x)$ as $k \uparrow \infty$.
This shows that $\lim_{n \uparrow \infty} {\mathbb P}_n(\Phi^0_n(f) \leq x) = {\mathbb P}(\Phi^0(f) \leq x)$.

Theorem 4 of \cite{CCK19} implies that the full-plane magnetization \eqref{def:full-plane-field} is conformally covariant. In particular, one has that
\begin{equation} \label{eq:conf-cov}
\Phi({\bf I}_{[-\alpha L,\alpha L]^2}) \stackrel{dist}{=} \alpha^{15/8} \Phi({\bf I}_{[-L,L]^2})
\end{equation}
where $\stackrel{dist}{=}$ denotes equality in distribution.
An equivalent way to express the conformal covariance of the magnetization field is presented in the next theorem, where, with a slight abuse of notation, we write $\Phi^0(x)$ even though $\Phi^0$ is not defined pointwise.
\begin{theorem}[Theorem 4.1 of \cite{cjn20}] \label{thm:conf-cov}
	For any $\lambda>0$, the field $\Phi^0_{\lambda}(x)=\Phi^0(\lambda x)$ given by
	\begin{eqnarray} \label{eq:conf-cov1}
	\Phi^0_{\lambda}(f) & = & \int_{{\mathbb R}^2} \Phi^0(\lambda x) f(x) dx \\
	& = & \int_{{\mathbb R}^2} \Phi^0(y) f(\lambda^{-1} y) \lambda^{-2} dy = \lambda^{-2} \Phi^0(f_{\lambda^{-1}}),
	\end{eqnarray}
	with $f_{\lambda^{-1}}(x) = f(\lambda^{-1}x)$ is equal in distribution to $\lambda^{-1/8}\Phi^0(x)$.
\end{theorem}

Theorem \ref{thm:conf-cov} is a special case of Theorem 1.8 of \cite{cgn15}, which shows that the distribution of the magnetization field transforms covariantly under any conformal map between any two simply connected domains. For more details, we refer the reader to Theorem 1.8 and Corollary 1.9 of \cite{cgn15}, as well as Section 4.2 of \cite{cjn20}.

\section{The near-critical scaling limit} \label{sec:near-critical-sl}

CME$_{16/3}$ describes the continuum scaling limit of the rescaled counting measures associated with \emph{critical} FK-Ising clusters.
As we have seen in the previous section, CME$_{16/3}$ provides, via equation \eqref{eq:cont-field}, a sort of continuum FK representation for the \emph{critical} magnetization field.
Surprisingly, CME$_{16/3}$, also plays a crucial role in the proof of exponential decay for the \emph{near-critical} magnetization field in \cite{cjn20}.
The tool that allows the use of CME$_{16/3}$ in the analysis of the near-critical scaling limit is the coupling presented in Corollary~\ref{new-coupling}.
%

Let $P_0$ denote the law of CME$_{16/3}$ in a bounded, simply connected domain $D$, and $E_0$ denote expectation with respect to $P_0$.
Equation \eqref{FK_external_field}, combined with Theorem~\ref{thm:jointconv} and with Corollary~3.8 of \cite{cgn15}, shows that in the near-critical regime, $\beta=\beta_c$ and $H=ha^{15/8}$, as $a \downarrow 0$, the distribution of the collection of FK clusters in $D$ has a limit $P_h$ defined by the Radon-Nikodym derivative
\begin{align} \label{RD-derivative}
\frac{dP_h}{dP_0} =& \frac{ \exp(\beta_ch \mu^0_{\mathcal{C}_b}({\mathbf I}_D))\Pi_{\mathcal{C}\in\mathscr{C}_D:\mathcal{C}\neq\mathcal{C}_b} \cosh(\beta_c h \mu^0_{\mathcal C}({\mathbf I}_D))}{E_0( \exp(\beta_ch \mu^0_{\mathcal{C}_b}({\mathbf I}_D)) \Pi_{\mathcal{C}\in\mathscr{C}_D:\mathcal{C}\neq\mathcal{C}_b} \cosh(\beta_c h \mu^0_{\mathcal C}({\mathbf I}_{D})))}\\
=& \frac{ \exp(\beta_ch \mu^0_{\mathcal{C}_b}({\mathbf I}_D)) \Pi_{\mathcal{C}\in\mathscr{C}_D:\mathcal{C}\neq\mathcal{C}_b} \cosh(\beta_c h \mu^0_{\mathcal C}({\mathbf I}_D))}{E(e^{\beta_c h M^0_D})},
\end{align}
where $M^0_D := \Phi^0_D({\mathbf I}_D)$ and we have used a near-critical scaling limit version of \eqref{pf-field} in the last equality.
We let $\{\mu^h_{\mathcal C}\}_{\mathcal{C}\in\mathscr{C}_D}$ denote a collection of measures distributed according to $P_h$; they are the scaling limit of the rescaled counting measures of the FK clusters in $D$ with free or wired boundary conditions in the near-critical regime.

In analogy with \eqref{def:cont-field}, we would like to define a near-critical magnetization field
\begin{equation} \label{eq:near-critical-field}
\Phi^h_D = \sum_{\mathcal{C}\in\mathscr{C}_D} \eta^h_{\mathcal{C}} \mu^h_{\mathcal{C}},
\end{equation}
where $\{\eta^h_{\mathcal C}\}_{\mathcal{C}\in\mathscr{C}_D}$ is a collection of independent random variables
such that $\eta_{\mathcal C}=1$ if ${\mathcal C}$ is a boundary cluster and otherwise $\eta_{\mathcal C}$ has distribution given by
\begin{align} \label{near-critical-signs}
\eta^h_{\mathcal C}=\begin{cases}
1 \text{ with probability } \frac{e^{\beta_c h \mu^h_{\mathcal C}({\mathbf I}_{D})}}{2 \cosh(\beta_c h \mu^h_{\mathcal C}({\mathbf I}_{D}))} = \tanh(\beta_c h \mu^h_{\mathcal C}({\mathbf I}_{D})) + \frac{1}{2}\big(1-\tanh(\beta_c h \mu^h_{{\mathcal C}}({\mathbf I}_{D}))\big) \\ -1 \text{ with probability } \frac{e^{-\beta_c h \mu^h_{\mathcal C}({\mathbf I}_{D})}}{2 \cosh(\beta_c h \mu^h_{\mathcal C}({\mathbf I}_{D}))} = \frac{1}{2}\big(1-\tanh(\beta_c h \mu^h_{\mathcal C}({\mathbf I}_{D}))\big)
\end{cases}
\end{align}

To make sense of \eqref{eq:near-critical-field} we follow the same strategy we used in the case of the critical magnetization \eqref{def:cont-field}: we define a near-critical $\varepsilon$-cutoff field
\begin{equation} \label{eq:cutoff-nc-field}
\Phi^h_{D,\varepsilon} := \sum_{\mathcal{C}\in\mathscr{C}_D: \diam(\mathcal{C})>\varepsilon} \eta^h_{\mathcal{C}} \mu^h_{\mathcal{C}}
\end{equation}
and show that it has a limit as $\varepsilon \downarrow 0$.
With these definitions, we can write a useful near-critical scaling limit version of \eqref{expectation-field} and \eqref{pf-field}, namely,
\begin{equation} \label{expecation-nc-field}
 E\big(g\big(\Phi^h_{D,\varepsilon}(f)\big)\big) = \frac{E\big( g\big(\Phi^0_{D,\varepsilon}(f)\big) e^{\beta_c h M^0_D} \big)}{E\big( e^{\beta_c h M^0_D} \big)} 
\end{equation}
for any suitable functions $f$ and $g$.

\begin{lemma} \label{Lemma:Sobolev_convergence_nc}
	Let $D$ be a bounded, simply connected domain.
	For every $\alpha >\frac{3}{2}$, the cutoff field $\Phi^h_{D,\varepsilon}$ converges as $\varepsilon \downarrow 0$ in mean in the Sobolev space $\mathcal{H}^{-\alpha}(D)$, in the sense that there exists a $\mathcal{H}^{-\alpha}(D)$-valued random field $\Phi_D^h$ such that
	\begin{equation}
	\lim\limits_{\varepsilon \downarrow 0} E\big(\| \Phi^h_{D,\varepsilon} - \Phi_D^h \|_{\mathcal{H}^{-\alpha}(D)}\big) = 0.
	\end{equation}
\end{lemma}
\begin{proof}
The proof is analogous to that of Lemma \ref{Lemma:Sobolev_convergence}, so we point out the only difference, which consists in replacing the $L^2$ bound \eqref{L2-bound} with the following $L^1$ bound, which uses \eqref{expecation-nc-field}, the Cauchy-Schwarz inequality, \eqref{L2-bound}
and the fact that $E\big(e^{\beta_c r M^0_D}\big)$ is finite for any $r \geq 0$ by the GHS inequality (see Proposition 3 and Lemma 4 of \cite{cjn_coupling}, and the proof of Proposition~3.5 in \cite{cgn15}):
\begin{eqnarray}
E(\| \Phi^h_{D,\varepsilon} - \Phi^h_{D,\varepsilon'} \|_{\mathcal{H}^{-\alpha}(D)})
& = & \frac{E\Big(\| \Phi^0_{D,\varepsilon} - \Phi^0_{D,\varepsilon'} \|_{\mathcal{H}^{-\alpha}(D)} e^{\beta_c h M^0_D}\Big)}{E\big(e^{\beta_c h M^0_D}\big)} \\
& \leq & \frac{\sqrt{E\Big(\| \Phi^0_{D,\varepsilon} - \Phi^0_{D,\varepsilon'} \|^2_{\mathcal{H}^{-\alpha}(D)} \Big) E\big(e^{2\beta_c h M^0_D}\big)}}{E\big(e^{\beta_c h M^0_D}\big)} \\
& \leq & \frac{\sqrt{E\big(e^{2\beta_c h M^0_D}\big)}}{E\big(e^{\beta_c h M^0_D}\big)} \Big(C' \sum_i \frac{1}{\lambda_i^{\alpha-1/2}}\Big)^{1/2} \varepsilon^{7/8}.
\end{eqnarray}
The rest of the proof is the same as the proof of Lemma \ref{Lemma:Sobolev_convergence}.
\end{proof}

Lemma \ref{Lemma:Sobolev_convergence_nc} shows that it makes sense to define the near-critical continuum magnetization field in a bounded domain $D$ as an element of ${\mathcal H}^{-\alpha}(D)$:
\begin{equation} \label{def:nc-field}
\Phi^h_D = \sum_{\mathcal{C}\in\mathscr{C}_D} \eta^h_{\mathcal{C}} \mu^h_{\mathcal{C}}
\end{equation}
where the right hand side is understood as the limit of $\Phi^h_{D,\varepsilon}$ as $\varepsilon \downarrow 0$.
Next, we will show that \eqref{def:nc-field} is the near-critical scaling limit of the lattice magnetization field.

\begin{theorem} \label{Thm:Sobolev_convergence_nc}
Consider an Ising model in the near-critical regime ($\beta=\beta_c$ and $H=ha^{15/8}$ for some $h>0$) with free or plus boundary condition on $D^a = D \cap a{\mathbb Z}^2$ where $D$ is a bounded, simply connected domain of ${\mathbb R}^2$. Then the lattice magnetization
\begin{equation}
\Phi^{a,H}_D := a^{15/8} \sum_{x\in D^a}\sigma_x\delta_x
\end{equation}
converges in distribution to $\Phi_D^h$ as $a \downarrow 0$. The convergence is in any Sobolev space $\mathcal{H}^{-\alpha}(D)$ with $\alpha >\frac{3}{2}$ in the topology induced by the norm $\| \cdot \|_{\mathcal{H}^{-\alpha}(D)}$.
\end{theorem}
\begin{proof}
The proof follows the same strategy as that of Theorem~\ref{Thm:Sobolev_convergence}, so we only point out the differences.
The first difference is that, in proving the tightness of the lattice magnetization, we need to replace the bounds leading to \eqref{second-moment-bound} with
\begin{align}
\limsup_{a \downarrow 0} E_{\beta_c,H}\Big( \| \Phi^{a,H}_D \|^2_{{\mathcal H}^{-\alpha'}(D)} \Big) = \limsup_{a \downarrow 0} \frac{E_{\beta_c,0}\Big( \| \Phi^{a,0}_D \|^2_{{\mathcal H}^{-\alpha'}(D)} e^{\beta_c H M^a_D} \Big)}{E_{\beta_c,0}\big(e^{\beta_c H M^a_D} \big)} \\
\qquad \leq  \frac{1}{E\big(e^{\beta_c h M^0_D} \big)} \limsup_{a \downarrow 0} \sqrt{E_{\beta_c,0}\Big( \| \Phi^{a,0}_D \|^4_{{\mathcal H}^{-\alpha'}(D)} \Big) E_{\beta_c,0}\big(e^{2\beta_c H M^a_D} \big)} \\
\qquad \leq \frac{\sqrt{E\big(e^{2\beta_c h M^0_D}\big)}}{E\big(e^{\beta_c h M^0_D}\big)} \limsup_{a \downarrow 0} \sqrt{E_{\beta_c,0}\Big[ \Big( \sum_i \frac{1}{\lambda_i^{\alpha'}} \big( \Phi_D^{a,0}(u_i) \big)^2 \Big)^2 \Big]} \\
\qquad \leq \frac{\sqrt{E\big(e^{2\beta_c h M^0_D}\big)}}{E\big(e^{\beta_c h M^0_D}\big)} \limsup_{a \downarrow 0} \sqrt{ \Big( \sum_i \frac{\| u_i \|^2_{L^{\infty}(D)}}{\lambda_i^{\alpha'}} \Big)^2 E_{\beta_c,0}\Big[ \big( M^a_D \big)^4 \Big]} \\
\qquad \leq \frac{\sqrt{E\big(e^{2\beta_c h M^0_D}\big)}}{E\big(e^{\beta_c h M^0_D}\big)} \Big( \sum_i \frac{\| u_i \|^2_{L^{\infty}(D)}}{\lambda_i^{\alpha'}} \Big) \limsup_{a \downarrow 0} \sqrt{ E_{\beta_c,0}\Big[ \big( M^a_D \big)^4 \Big]} \\
\qquad \leq \frac{\sqrt{E\big(e^{2\beta_c h M^0_D}\big)}}{E\big(e^{\beta_c h M^0_D}\big)} c^2 \Big(\sum_i \frac{1}{\lambda_i^{1+\epsilon/2}}\Big) \limsup_{a \downarrow 0} \sqrt{ E_{\beta_c,0}\Big[ \big( M^a_D \big)^4 \Big]}, \label{second-moment-bound-nc}
\end{align}
where we have used \eqref{expectation-field}, the Cauchy-Schwarz inequality and \eqref{second-moment-bound}. The last expression is finite because of Weyl's law \cite{weyl} and the boundedness of the fourth moment of the renormalized magnetization (see, e.g., Corollary~3.8 of \cite{cgn15}).

In order to show that all subsequential limits coincide with the field $\Phi^h_D$ defined by \eqref{def:nc-field}, we let $\tilde{\Phi}^h_D$ denote any such subsequential limit obtained from a converging sequence $\Phi^{a_k,H}_D$. We then use the coupling in Corollary \ref{new-coupling} to write
\begin{equation}
\Phi^{a_k,H}_D = \Phi^{a_k,H}_{D,\varepsilon} + \sum_{i : \diam({\mathcal C}_i^{a_k})\leq\varepsilon} \eta_i \mu^{a_k}_i,
\end{equation}
where $\Phi^{a_k,H}_{D,\varepsilon} := \sum_{i : \diam({ \mathcal C}_i^{a_k})>\varepsilon} \eta_i \mu^{a_k}_i$.

The argument then proceeds as in the proof of Theorem~\ref{Thm:Sobolev_convergence}, but with the bounds leading to \eqref{zero-limit} replaced by
\begin{align}
E_{\beta_c,H}(\| \tilde{X}_{\varepsilon} \|_{\mathcal{H}^{-\alpha}(D)}) \leq \limsup_{a_k \downarrow 0} E_{\beta_c,H}\Big(\Big\| \sum_{i : \diam({\mathcal C}_i^{a_k})\leq\varepsilon} \eta_i \mu^{a_k}_i({\mathbf I}_{D}) \Big\|_{\mathcal{H}^{-\alpha}(D)}\Big) \\
\qquad = \limsup_{a_k \downarrow 0} \frac{E_{\beta_c,0}\Big(\Big\| \sum_{i : \diam({\mathcal C}_i^{a_k})\leq\varepsilon} \eta_i \mu^{a_k}_i({\mathbf I}_{D}) \Big\|_{\mathcal{H}^{-\alpha}(D)} e^{\beta_c H M^a_D} \Big)}{E_{\beta_c,0}\big(e^{\beta_c H M^a_D}\big)} \\
\qquad \leq \frac{E\big(e^{2\beta_c h M^0_D} \big)}{E\big(e^{\beta_c h M^0_D}\big)} \limsup_{a_k \downarrow 0} E_{\beta_c,0}\Big(\Big\| \sum_{i : \diam({\mathcal C}_i^{a_k})\leq\varepsilon} \eta_i \mu^{a_k}_i({\mathbf I}_{D}) \Big\|^2_{\mathcal{H}^{-\alpha}(D)} \Big) \\
\qquad \leq \frac{E\big(e^{2\beta_c h M^0_D} \big)}{E\big(e^{\beta_c h M^0_D}\big)} C' \varepsilon^{7/4} \sum_i \frac{1}{\lambda_i^{\alpha-1/2}} \rightarrow 0 \text{ as } \varepsilon \to 0, \label{zero-limit-nc}
\end{align}
where we have used \eqref{expectation-field} and the Cauchy-Schwarz inequality, as before, and \eqref{small-diameter} and \eqref{zero-limit} in the last line. We then have that
\begin{align}
E(\| \Phi^{h}_D - \tilde{\Phi}^{h}_{D} \|_{\mathcal{H}^{-\alpha}(D)}) \leq
E(\| \Phi^{h}_D - \Phi^{h}_{D,\varepsilon} \|_{\mathcal{H}^{-\alpha}(D)}) +
E(\| \tilde{X}_{\varepsilon} \|_{\mathcal{H}^{-\alpha}(D)}) \\
\qquad \leq E(\| \Phi^{h}_D - \Phi^{h}_{D,\varepsilon} \|_{\mathcal{H}^{-\alpha}(D)}) + \frac{E\big(e^{2\beta_c h M^0_D} \big)}{E\big(e^{\beta_c h M^0_D}\big)} C' \varepsilon^{7/4} \sum_i \frac{1}{\lambda_i^{\alpha-1/2}}. \label{bound-mean-nc}
\end{align}

Markov's inequality now implies that, for any $\delta>0$,
$P\big( \| \Phi^h_D - \tilde{\Phi}^h_D \|_{\mathcal{H}^{-\alpha}(D)} > \delta \big)$ can be made arbitrarily small using the bound \eqref{bound-mean-nc} together with \eqref{zero-limit-nc} and Lemma \ref{Lemma:Sobolev_convergence_nc}, and taking $\varepsilon$ small.
This concludes the proof of uniqueness of subsequential limits and of the theorem.
\end{proof}

As in the zero external field case (see \eqref{def:full-plane-field} and the discussion at the end of Section~\ref{sec:m-field}), one can consider a full-plane version of the magnetization field:
\begin{equation} \label{def:full-plane-field-nc}
\Phi^h = \sum_{\mathcal{C}\in\mathscr{C}} \eta^h_{\mathcal{C}} \mu^h_{\mathcal{C}}.
\end{equation}
The main ingredients in the construction of $\Phi^h$ are RSW, FKG and stochastic domination. Using those ingredients, for any $n \in {\mathbb N}$ and $\varepsilon>0$,
one can find an $m \in {\mathbb N}$ sufficiently large and a coupling between an Ising configuration on $a{\mathbb Z}^2$ and an Ising configuration in $[-m,m]^2 \cap a{\mathbb Z}^2$ with plus boundary condition such that the two configurations coincide inside $[-n,n]^2$ with probability at least $1-\varepsilon$. A discussion of this construction can be found in the proof of Lemma 4.1 of \cite{cgn16}.

Contrary to the critical magnetization field $\Phi^0$, the near-critical field $\Phi^h$ is not conformally covariant and the considerations at the very end of Section \ref{sec:m-field} do not apply to it. Indeed, $\Phi^h$ exhibits exponential decay of correlations, or a \emph{mass gap} in the language of field theory. This important result, proved in \cite{cjn20}, is too complex to discuss in detail and provide a complete proof here. However, below we comment briefly on it and its proof.

Let $\langle \sigma_{x};\sigma_{y} \rangle_{\beta,H}$ denote the Ising truncated two-point function, i.e.,
\begin{equation} \label{def:truncated2pf}
\langle \sigma_{x};\sigma_{y} \rangle_{\beta,H} := \langle \sigma_{x}\sigma_{y} \rangle_{\beta,H} - \langle \sigma_{x} \rangle_{\beta,H} \langle \sigma_{y} \rangle_{\beta,H}.
\end{equation}
The exponential decay of correlations in the near-critical field $\Phi^h$ is proved in \cite{cjn20} as a consequence of the following result.

\begin{theorem}[Theorem 1.1 of \cite{cjn20} and Theorem 1 of \cite{cjn_coupling}] \label{thm:exp-decay}
Consider the Ising model on $a{\mathbb Z}^2$ at inverse temperature $\beta=\beta_c$ and with external field $H=ha^{15/8}$.
There exists $B_0, B_1,C_0, C_1 \in (0,\infty)$ such that, for any $a \in (0,1]$ and $h \in (0,a^{-15/8}]$,
\begin{equation} \label{eq:exp-decay}
C_1 a^{1/4} h^{2/15} e^{-B_1 h^{8/15} |x-y|} \leq \langle \sigma_{x};\sigma_{y} \rangle_{\beta,H} \leq C_0 a^{1/4} |x-y|^{-1/4} e^{-B_0 h^{8/15} |x-y|}
\end{equation}
for any $x,y \in a{\mathbb Z}^2$.
\end{theorem}

The exponential decay of the Ising truncated two-point function for $H>0$ was first proved in \cite{LP68}, but Theorem \ref{thm:exp-decay} is the first result to provide the correct rate of decay as a function of the external field in the near-critical regime, combined with the correct power law behavior at short distances. These features are crucial in order to obtain a meaningful bound in the scaling limit.
Using Theorem \ref{thm:exp-decay}, one has the following result, also proved in \cite{cjn20}.

\begin{theorem}[Theorem 1.4 of \cite{cjn20}] \label{thm:exp-decay-field}
For any $f,g \in C^{\infty}_0$, we have
\begin{equation} \label{eq:exp-decay-field}
|\text{\emph{Cov}}(\Phi^h(f),\Phi^h(g))| \leq C_0 \iint_{{\mathbb R}^2 \times {\mathbb R}^2} |f(x)| |g(y)| |x-y|^{-1/4} e^{-B_0 h^{8/15} |x-y|} dx dy,
	\end{equation}
where $B_0$ and $C_0$ are as in Theorem \ref{thm:exp-decay}.
\end{theorem}

Although the near-critical magnetization field is not conformally covariant like the critical one, it still possesses interesting scaling properties, as shown by the next result.

\begin{theorem}[Theorem 4.2 of \cite{cjn20}]\label{thmscl2}
	For any $\lambda>0$ and $h>0$, the field $\lambda^{1/8}\Phi^{h}(\lambda x)$ is equal in distribution to $\Phi^{\lambda^{15/8}h}(x)$.
\end{theorem}

The following observation may be useful to interpret Theorem \ref{thmscl2}.
As discussed at the end of Section \ref{sec:m-field}, in the zero-field case, $\Phi^{0}(\lambda x)$ is equal in distribution to $\lambda^{-1/8}\Phi^{0}(x)$ in the sense that, with the change of variables $z=\lambda x$,
\[\int \Phi^{0}(z)f(z)dz \stackrel{dist}{=} \int \lambda^{-1/8} \Phi^{0}(x) f(\lambda x) \lambda^2 dx \stackrel{dist}{=} \lambda^{15/8} \int \Phi^{0}(x) f(\lambda x) dx\]
for any $f\in C^{\infty}_0(\mathbb{R}^2)$, where the equalities are in distribution.
In the non-zero-field case, provided that $\tilde{h}=\lambda^{-15/8}h$, using Theorem \ref{thmscl2} one obtains an analogous relation as follows:
\begin{eqnarray*}
	\int\Phi^{\tilde{h}}(z)f(z)dz&=&\int\Phi^{\lambda^{-15/8}h}(\lambda x)f(\lambda x) \lambda^2 dx\\
	&=&\lambda^{15/8}\int\lambda^{1/8}\Phi^{\lambda^{-15/8}h}(\lambda x)f(\lambda x) dx\\
	&=&\lambda^{15/8}\int\Phi^{h}(x)f(\lambda x) dx.
\end{eqnarray*}

We now consider the field $\Phi_D^h$ in a simply connected domain not equal to $\mathbb C$ and a conformal map $\phi: D \to \tilde{D}$ with inverse $\psi=\phi^{-1}:\tilde{D} \rightarrow D$. The pushforward by $\phi$ of $\Phi^0_{D}$ to a generalized field on $\tilde{D}$ is described explicitly in Theorem 1.8 of \cite{cgn15}. The generalization to $\Phi^h$, implicit in \cite{cgn16}, is stated explicitly in the next theorem, taken from \cite{cjn20}, where we introduce the non-constant magnetic fields $h(z)$ and $\tilde{h}(x)$.
\begin{theorem}[Theorem 4.3 of \cite{cjn20}]\label{thmcon}
	The field $\Phi^{h}_{D,\psi}(x):=\Phi^{h}_{D}\left(\psi(x)\right)$ on $\tilde{D}$ is equal in distribution to the field $|\psi^{\prime}(x)|^{-1/8}\Phi^{\tilde{h}}_{\tilde{D}}(x)$ on $\tilde{D}$, where $\tilde{h}(x)=|\psi^{'}(x)|^{15/8}h(\psi(x))$.
\end{theorem}

Combining the scaling properties of $\Phi^h$ with Theorem \ref{thm:exp-decay-field} leads to the precise power-law behavior of the \emph{mass} ${\mathbf m}(\Phi^h)$ of $\Phi^h$, defined as the supremum over all $m$ such that, for all $f,g \in C^{\infty}_0({\mathbb R}^2)$ and some $C_m(f,g)<\infty$,
\begin{equation}
|\text{Cov}(\Phi^h(f),\Phi^h(T^u g))| \leq C_m(f,g) e^{-mu},
\end{equation}
where $(T^u g)(x_0,x_1) = g(x_0-u,x_1)$.

\begin{corollary}[Corollary 1.6 of \cite{cjn20}]\label{thm:mass}
${\mathbf m}(\Phi^h)=Ch^{8/15}$ for some $C \in (0,\infty)$ and all $h$.
\end{corollary}

\begin{remark}
As mentioned above, the power law behavior ${\mathbf m}(\Phi^h)=Ch^{8/15}$ of the mass ${\mathbf m}(\Phi^h)$ follows from the scaling properties of the field $\Phi^h$. The fact that $C\in(0,\infty)$ is equivalent to $0<{\mathbf m}(\Phi^h)<\infty$, where the lower bound ${\mathbf m}(\Phi^h)>0$ is an immediate consequence of Theorem \ref{thm:exp-decay-field}. As explained in the proof of Corollary 1.6 of \cite{cjn20}, assuming that ${\mathbf m}(\Phi^h)=\infty$ leads to a contradiction, which proves that ${\mathbf m}(\Phi^h)<\infty$.
\end{remark}

We conclude this section and the paper with a sketch of the main ideas of the proof of the upper bound of Theorem \ref{thm:exp-decay}.
In the discussion below we assume that the reader has some familiarity with (FK) percolation and we use extensively the FK representation of the Ising model described in Section \ref{sec:Ising&FK}, and in particular the couplings discussed in Section \ref{sec:bounded}. In this context, the notation we use should be self-explanatory (e.g., we will use $\{x\longleftrightarrow y\centernot\longleftrightarrow g\}$ to denote the event that vertices $x$ and $y$ are in the same FK cluster and that that cluster is not connected to the ghost).

The first step of the proof of exponential decay consists in writing
\begin{eqnarray}
\langle\sigma_x;\sigma_y\rangle_{\beta_c,H} & = & P_{\beta_c,H}^{FK}(x\longleftrightarrow y)- P_{\beta_c,H}^{FK}(x\longleftrightarrow g)P_{\beta_c,H}^{FK}(y\longleftrightarrow g) \\
& = &  P_{\beta_c,H}^{FK}(x\longleftrightarrow y \centernot\longleftrightarrow g) + \text{Cov}_{\beta_c,H}^{FK}({\mathbf I}_{\{x\longleftrightarrow g\}}, {\mathbf I}_{\{y\longleftrightarrow g\}}), \label{eqdifference}
\end{eqnarray}
where $\text{Cov}_{\beta_c,H}^{FK}$ denotes covariance with respect to the FK measure $P_{\beta_c,H}^{FK}$ on $a\mathbb{Z}^2$.

Letting $B(x,L)$ denote the square centered at $x$ of side length $2L$ and writing
\begin{align*}
A^{near}_x:=\{&\text{there exists an FK-open path from } x, \text{ within }B(x,|x-y|/3), \text { to some } \\
&w \text{ with the edge from } w \text{ to } g \text{ open}  \}
\end{align*}
and $A^{far}_x:=\{x\longleftrightarrow g\}\setminus A^{near}_x$, so that $\{x\longleftrightarrow g\} = A^{near}_x \cup A^{far}_x$, the covariance in \eqref{eqdifference}
can be written as a sum of four covariances and $\langle\sigma_x;\sigma_y\rangle_{\beta_c,H}$ as a sum of five terms.
Bounding four of these five terms reduces to showing that, when $H=ha^{15/8}$,
\begin{equation} \label{eqstar}
P_{\beta_c,H}^{FK}(g \centernot\longleftrightarrow x \longleftrightarrow \partial B(x,|x-y|/3)) \leq \tilde C(h) a^{1/8} e^{-\hat C(h)|x-y|}.
\end{equation}
The remaining term, $\text{Cov}_{\beta_c,H}^{FK}(1_{A_x^{near}}, 1_{A_y^{near}})$, needs a separate argument and will be discussed later.

Focusing for now on \eqref{eqstar}, the power law part of the upper bound comes from a 1-arm argument, while the exponential part requires a more sophisticated argument that makes use of CME$_{16/3}$ coupled to CLE$_{16/3}$ (see Section \ref{sec:CLE&CME}) as well as a stochastic domination theorem by Liggett, Schonmann and Stacey \cite{LSS97}.

\begin{figure}
	\begin{center}
		\includegraphics[width=0.8\textwidth]{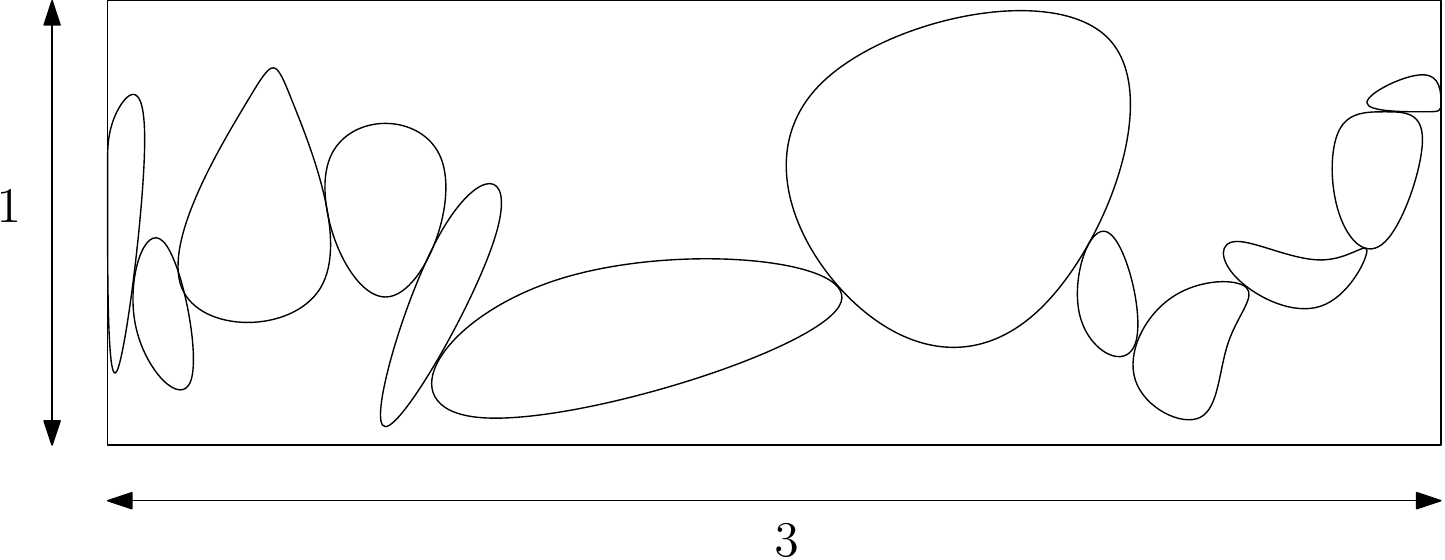}
		\caption{An illustration of the event described in the sketch of the proof of Theorem \ref{thm:exp-decay} representing a chain of large FK clusters crossing a rectangle.}\label{fig:E(K)}
	\end{center}
\end{figure}

Roughly speaking, what we use of the coupled CLE$_{16/3}$ and CME$_{16/3}$ is the fact that, for $K$ large, a realization inside the rectangle $\Lambda_{3N,N}:=[0,3N]\times[0,N]$ is likely to contain a chain of not more than $K$ touching loops that cross the rectangle in the long direction, with the first loop touching one of the short sides of the rectangle and the last loop touching the opposite side (see Fig.~\ref{fig:E(K)}). Moreover, the ``area'' of each of the continuum clusters associated to the loops in the chain is likely to be larger than $N^{15/8}/K$, with the probability of the event just described going to one as $K \to \infty$. Back on the lattice this implies that, inside $\Lambda_{3N,N}$, one can find with high probability a chain of FK clusters one lattice spacing away from each other and crossing the rectangle. Moreover, such clusters will, with high probability, have sizes larger than $N^{15/8}a^{-15/8}/K$, which in turn implies that there is a high probability that they are each connected to $g$ if the external field is $H=ha^{15/8}$ and $N$ is large.

Combining all of the above, with the help of the FKG inequality, one can show that, with high probability, a large annulus contains a circuit of FK clusters one lattice spacing away from each other, each connected to $g$, such that the circuit disconnects the inner square of the annulus from infinity. We call such an annulus \emph{good}.

In order to complete the proof of \eqref{eqstar}, one covers the plane with large overlapping annuli in such a way that their inner squares tile the plane. For each such annulus, the event that it is good happens with high probability. We would like to conclude that good annuli percolate, but the annuli are overlapping, so the events are not independent. To deal with this, one can use a stochastic domination result due to Liggett, Schonmann and Stacey \cite{LSS97}. Now, percolation of good annuli implies that the probability that $x$ is surrounded by a circuit of good annuli contained in a square $B(x,L)$ of size $2L$ centered at $x$ is close to one, exponentially in $L$. But because of planarity, if $x$ is surrounded by a circuit of good annuli contained in $B(x,L)$, the event $\{ g \centernot\longleftrightarrow x \longleftrightarrow \partial B(x,L) \}$ cannot happen. This provides the desired exponential bound.

The remaining term can be written as
\begin{eqnarray*}
	\text{Cov}_{\beta_c,H}^{FK}({\mathbf I}_{A_x^{near}}, {\mathbf I}_{A_y^{near}}) & = &
	P_{\beta_c,H}^{FK}(A^{near}_x \cap A^{near}_y) - P_{\beta_c,H}^{FK}(A^{near}_x)P_{\beta_c,H}^{FK}(A^{near}_y) \\
	& = & P_{\beta_c,H}^{FK}(A^{near}_y) \big[ P_{\beta_c,H}^{FK}(A^{near}_x | A^{near}_y) - P_{\beta_c,H}^{FK}(A^{near}_y) \big].
\end{eqnarray*}
A 1-arm argument provides a polynomial upper bound of order $a^{1/8}$ for $P_{\beta_c,H}^{FK}(A^{near}_y)$.
The first step in dealing with the remaining factor consists in showing that $P_{\beta_c,H}^{FK}(A^{near}_x | A^{near}_y)$ is smaller than the probability of the event $A^{near}_x$ with wired boundary condition on $B(x,2|x-y|/3)$. A key ingredient in proving this fact is the monotonicity of the Radon-Nikodym derivative of a suitable conditional FK measure in $B(x,2|x-y|/3)$ with respect to the FK measure in $B(x,2|x-y|/3)$ with wired boundary condition. The remaining step consists in showing that the probability of $A^c_x$ is not affected much by the boundary condition in $B(x,2|x-y|/3)$. This completes our sketch of the proof of the upper bound of Theorem \ref{thm:exp-decay}.

\bigskip

\noindent{\bf Acknowledgments.} The authors thank Wouter Kager for Figures \ref{phase_transition} and \ref{fig:loops}, and an anonymous referee for useful comments. The first author thanks Fran\c{c}ois Dunlop, Ellen Saada, and Alessandro Giuliani for organizing the IRS 2020 conference and for the invitation to give a presentation and to contribute to the conference proceedings. The research of the second author was partially supported by  NSFC grant 11901394 and that of the third author by US-NSF grant DMS-1507019.

\end{document}